\documentclass{amsart}
\usepackage{macros}
\standardsettings
\colorcommentstrue

\newcommand{\emb}{\Phi}

\newcommand{\numberq}{T}

\renewcommand{\ss}{\mathbf s}

\renewcommand{\Neur}{\N_0}

\newcommand{\aff}{{\mathrm{aff}}}
\DeclareMathOperator{\Span}{Span}

\renewcommand{\O}{\stopcompiling}

\renewcommand{\dist}{\operatorname{dist}}

\newcommand{\eq}[2]{
\begin{equation}
\label{eq:#1}
{#2}
\end{equation}
}
\newcommand{\equ}[1]{\eqref{eq:#1}}

\renewcommand{\HD}{\dim}

\theoremstyle{theorem}
\maketheorem{affinecorollary}{Affine Corollary}
\theoremstyle{definition}
\maketheorem{affinecorollaryD}{Affine Corollary}

\draftfalse

\long\def\commkeith#1{\textcolor{cyan}{#1}}

\begin{document}
\title{Intrinsic Diophantine approximation on manifolds: General theory}

\authorlior\authordmitry\authorkeith\authordavid

%\subjclass[2010]{Primary }
%\keywords{}
%\date{}
%\dedicatory{}

\begin{Abstract}
%We consider the question of how well points in a manifold $M\subseteq \mathbb R^d$ can be approximated by rational points of $\mathbb Q^d\cap M$. This contrasts with the more common setup of approximating points in $M$ by all rational points in $\mathbb Q^d$. We give an upper bound on the intrinsic analogue of the exponent of irrationality of a nondegenerate manifold, which is in some cases sharp. The main tool in the proof is an analogue of the Simplex Lemma for rational points all lying on the same manifold.

We investigate the question of how well points on a nondegenerate $k$-dimensional submanifold $M \subseteq \mathbb R^d$ can be approximated by rationals also lying on $M$, establishing an upper bound on the ``intrinsic Dirichlet exponent" for $M$. We show that relative to this exponent, the set of badly intrinsically approximable points is of full dimension and the set of very well intrinsically approximable points is of zero measure. Our bound on the intrinsic Dirichlet exponent is phrased in terms of an explicit function of $k$ and $d$ which does not seem to have appeared in the literature previously. It is shown to be optimal for several particular cases. The requirement that the rationals lie on $M$ distinguishes this question from the more common context of (ambient) Diophantine approximation on manifolds, and necessitates the development of new techniques. Our main tool is an analogue of the Simplex Lemma for rationals lying on $M$ which provides new insights on the local distribution of rational points on nondegenerate manifolds.
\end{Abstract}

\maketitle

%\tableofcontents

%\draftnewpage
\section{Introduction and motivation}

In its classical form, the field of Diophantine approximation addresses questions regarding how well points $\xx \in \R^d$ can be approximated by rational points, %$\rr \in \Q^d$, 
where the quality $\| \xx - \rr \|$ of a rational approximation $\rr \in \Q^d$ is compared with the size of the denominator of $\rr$. The most fundamental theorem in the field is Dirichlet's theorem, dating back to 1842, which establishes a rate of approximation which holds for \emph{every} point.\Footnote{It was recently pointed out to us by Bugeaud that in the case of $\R$, this result is actually much older, coming directly from the theory of continued fractions (see e.g. \cite[displayed equation on p.28]{Legendre}).} The full significance of this result was realized two years later when Liouville established that certain real numbers, namely quadratic irrationals, do not admit a better rate of approximation. Liouville's result was later generalized to higher dimensions by Perron \cite{Perron}; together, these results show that Dirichlet's theorem is optimal in every dimension, in a sense to be made rigorous below.

Questions related to Diophantine approximation can be asked in a much broader context. Consider a closed subset $M$ of a complete metric space $(X,\dist)$, a countable subset $\QQ \subset X$ %whose closure contains $M$, 
and a \emph{height} function $H \colon \QQ \to (0,\infty)$. Modifying the terminology recently introduced in \cite{FishmanSimmons5}, we will refer to such a collection as a \emph{Diophantine triple}, and denote it by $(M,\QQ,H)$. Given {such} a triple, we can then look for a function $\psi \colon (0,\infty) \to (0,\infty)$ such that for every $\xx \in M$, there exists a constant $C_{\xx}$ and a sequence $\rr_n \in \QQ$ ($n\in\N$) satisfying
\[
\rr_n \to \xx \quad \text{and} \quad \dist(\xx,\rr_n) \leq C_{\xx}\psi\big(H(\rr_n)\big) %\text{\commkeith{deleted 
{\text{ for all } n} .
\]
{We will call} such a function $\psi$ a \emph{Dirichlet} function for approximation of points in $M$ by $\QQ$. {(Note that $M$ has to be contained in the closure of $\QQ$ in order for a Dirichlet function to exist.)}
%If moreover the constant $C_{\xx}$ can be chosen independently of $\xx$, we say that $\psi$ is a \emph{uniform} Dirichlet function. 
Using this terminology, and defining the height $H \colon \Q^d \to (0,\infty)$ as $H(\pp/q) := q$, where $\pp\in\Z^d$ and $q\in\N$ are chosen so that $\pp/q$ is in reduced form, it follows from Dirichlet's theorem that %{implies that}
%can be stated as follows:
%\Footnote{Actually, Dirichlet's theorem is slightly stronger than \eqref{equationdir}, as it says that the constant $C_\xx$ is not only uniform but equal to 1.}
\begin{equation}
\label{equationdir}
\psi_{1 + 1/d} \text{ is a 
%uniform 
Dirichlet function for the triple } (\R^d,\Q^d,H),
\end{equation}
where here and hereafter we write 
\begin{equation}
\label{psi_c}
\psi_c(t) := 1/t^c 
\end{equation}
{and use the max norm to define distance on $\R^d$.}
%In our terminology, this says that with respect to the height function $H\big(\frac{\pp_n}{q_n}\big) = q_n$, the function $\psi(q) := q^{-(1+1/d)}$ is a uniform Dirichlet function. 
Once a Dirichlet function has been identified, a natural question is whether it is optimal in the following sense. Call a Dirichlet function $\psi$ \emph{optimal} if there does not exist another Dirichlet function $\phi$ for which $\frac{\phi(t)}{\psi(t)} \to 0$ as $t \to \infty$, i.e. that we cannot find a faster decaying Dirichlet function. The optimality of a Dirichlet function $\psi$ is {clearly} implied by{\Footnote{{and in many cases equivalent to, see \cite[\62]{FishmanSimmons5} for a thorough discussion}}
the set $\BA(M,\QQ,H,\psi)$ of \emph{$\psi$-badly approximable} points being nonempty, where we define
$$
\BA(M,\QQ,H,\psi):=\big\{\xx \in M: \exists\,c_\xx > 0\text{ such that }\dist(\xx,\rr) \geq c_\xx \psi\big(H(\rr)\big)\ \forall\,\rr \in \QQ\big\}.
$$
Thus the aforementioned result of Perron shows that $\psi_{1 + 1/d}$ is an 
%uniform 
optimal Dirichlet function for the triple $(\R^d,\Q^d,H)$. Note that Perron's result has been later strengthened by Schmidt, who showed that the set \eq{defbad}{
\BA_d :=\BA(\R^d,\Q^d,H,\psi_{1 + 1/d})} of badly approximable vectors is of full Hausdorff dimension, and, moreover, is a winning set. As a way to interpret these results, for a Diophantine triple $(M,\QQ,H)$ let us define the \emph{Dirichlet exponent} $\delta(M,\QQ,H)$ to be the supremum of $c > 0$ such that $\psi_c$ is a Dirichlet function for $(M,\QQ,H)$. The theorems of Dirichlet and Perron then imply that the Dirichlet exponent of $(\R^d,\Q^d,H)$ is equal to $1 + 1/d$.}

\smallskip

{Another class of examples of Diophantine triples is provided by the field of \emph{(ambient) Diophantine approximation on manifolds} (see, for instance, \cite{Beresnevich_Khinchin, Beresnevich_BA, KleinbockMargulis2}). Namely, let $M$ be a smooth submanifold of $\R^d$, and consider the triple $(M,\Q^d,H)$. Clearly, by \eqref{equationdir}, %, by Dirichlet's theorem, 
%the function 
$\psi_{1 + 1/d}$ is still a Dirichlet function. 
%for $(M, \Q^d,H_{\rm \new{std}})$. 
On the other hand} it is easy to choose a manifold $M$, for example a rational affine subspace of $\R^d$, such that every point of $M$ admits a much better rate of approximation than the rate given by $\psi_{1 + 1/d}$. In order to rule out such behavior one is led to {impose a nondegeneracy condition (see Definition \ref{definitionnondegenerate}). And indeed, recently
Beresnevich \cite{Beresnevich_BA} proved
%\Footnote{In \cite{Beresnevich_BA} it is assumed that $M$ is real-analytic, but the method seems to work in the smooth category as well. Also, earlier this was proven by D.\ A.\ Badziahin and S.\ L.\ Velani \cite{BadziahinVelani2} for $C^2$ nondegenerate planar curves. } 
that %such 
for any real-analytic nondegenerate submanifold $M$ of $\R^d$, the set 
$$\BA(M,\Q^d,H,\psi_{1 + 1/d}) = M\cap\BA_d$$ has full Hausdorff dimension, thereby showing the optimality of $\psi_{1 + 1/d}$; earlier this was established by Badziahin and Velani \cite{BadziahinVelani2} for smooth nondegenerate planar curves. Consequently, the \emph{ambient Dirichlet exponent}
$$\delta^{\rm amb}_M := \delta(M,\Q^d,H)$$
is equal to $1 + 1/d$ whenever $M$ is nondegenerate.
} %The optimality of $\psi_{1 + 1/d}$ for the ambient approximation on $M$ follows.

% Here, we say that $\xx \in X$ is \emph{badly approximable} with respect to $\psi$ if there exists a constant $\espilon = \epsilon(\xx) > 0$ such that for every $\rr \in \QQ$, $\dist(\xx,\rr) \geq \epsilon \psi(H(\rr))$. We denote by $\BA(\psi)$ the collection of points in $X$ which are badly approximable with respect to $\psi$.
%As mentioned earlier, badly approximable points in $\R$ were first exhibited by Liouville.
%\begin{lemma}[Liouville]
%\label{lemmaliouville}
%For every algebraic number $x \in \R$ of degree $n$, there exists a constant $C$ such that for every rational $\frac{p}{q}$, we have $$\big| x - \frac{p}{q}\big| \geq \frac{C}{q^n}.$$ In %particular, taking $n=2$ shows that any irrational quadratic is badly approximable with respect to $\psi(q) = q^{-2}$. 
%\end{lemma}

\smallskip

{The goal of this paper is to develop the theory of \emph{intrinsic approximation} on manifolds; that is, 
%investigate the degree to which points on a smooth submanifold $M$ of $\R^d$ can be approximated by the rational points lying \emph{on} $M$. Explicitly, 
we will set $(X,\dist)$ to be $\R^d$ equipped with the max norm, $M$ a smooth submanifold of dimension $k \leq d$, $\QQ := \Q^d \cap M$, and $H(\pp/q) = q$ as before.}
%$H(\frac{\pp}{q}) = q$. 
So we are interested in the triple $(M, \Q^d \cap M, H)$. 
%Because of the requirement that the rational points lie on $M$, we call this notion of approximation \emph{intrinsic approximation}. This is in contrast to the usual practice of approximating $\xx \in M$ by any rational $\rr \in \Q^d$, a field called \emph{(ambient) Diophantine approximation on manifolds} (see, for instance, \cite{Beresnevich_Khinchin, Beresnevich_BA, KleinbockMargulis2}). 
The field of intrinsic approximation has seen a lot of recent activity in many diverse contexts; see e.g.\ \cite{BFR, FishmanSimmons1, GGN1, GGN2, KleinbockMerrill, Schmutz}. {Most recently, 
%as well as our companion paper 
in the companion paper \cite{FKMS2} we have obtained definitive results for $M$ being a nonsingular rational quadric hypersurface of $\R^d$ containing a dense set of rational points. In particular, it is proved there \cite[Theorem 5.1]{FKMS2} that for such $M$, $\psi_1$ is a Dirichlet function for intrinsic approximation, and, moreover, it is optimal because \cite[Theorem 4.5]{FKMS2} when $\psi = \psi_1$, the set 
$$
\BA_M(\psi) :=\BA(M,\Q^d\cap M,H,\psi)$$
has full Hausdorff dimension. Earlier this was established in \cite{KleinbockMerrill} for the unit sphere $M = S^{d-1}$.}

{Now let $M$ be an arbitrary $k$-dimensional nondegenerate smooth submanifold of $\R^d$. 
Is it possible to establish similar results?}
%Returning now to the context of \emph{intrinsic} approximation, namely the triple $(M, \Q^d \cap M, H)$, and given a $k$-dimensional submanifold $M \subset \R^d$, the first question we ask is whether one can find an (optimal) Dirichlet function $\psi$ for approximating points on $M$ intrinsically by rationals on $M$. 
Clearly {for that one needs some information on the set of rational points inside $M$}. 
%, much more depends on the manifold itself than in the theory of ambient approximation. 
As an extreme case, %it is easy to find manifolds $M$ for which 
%\new{intrinsic approximation is vacuous
{Dirichlet functions do not exist if $M \cap \Q^d = \emptyset$}. However, the following example shows that even if $\Q^d \cap M$ is dense in $M$, the ``quantitative denseness" which determines Diophantine properties {might depend} on $M$. 

\begin{example}
\label{exampletrivial}
Fix $n \geq 2$, let $\Phi_n: \R \to \R^2$ be the map $x \mapsto (x,x^n)$, {and let $C_n := \Phi_n(\R)$} denote the image curve in $\R^2$. Then it is easy to see that $\Q^2 \cap C_n = \Phi_n(\Q)$ and $H\big(\Phi_n(p/q)\big) = q^n$. Since $\psi_{1 + 1/d}$ is an optimal Dirichlet function on $\R$, this implies that $\psi_{2/n}$ is an optimal %uniform 
Dirichlet function for intrinsic approximation on $C_n$.
\end{example}

%In particular, this example shows that no function can be an optimal Dirichlet function for every curve in $\R^2$. 
Note that as $n \to \infty$, the functions $\psi_{2/n}$ decay more and more slowly\Footnote{{Note that} the existence of these examples does not rule out the possibility that {some function decaying slower than all of the functions $\psi_{2/n}$, e.g.\ $\psi(t) = 1/\log(t)$}, is a Dirichlet function for every nondegenerate manifold whose intrinsic rationals are dense. It would be interesting to investigate this question further.}, {yet none of them 
%functions $\psi_{2/n}$ 
decays} faster than $\psi_1$. So we can ask: does there exist a nondegenerate curve %(see Definition \ref{definitionnondegenerate}) 
{with an intrinsic Dirichlet function decaying} faster than $\psi_1$? {Our first theorem shows that} 
%in fact 
the answer is no. %Before stating the theorem, let us make a few remarks about statements of this type. We {are going to} 
{More generally, it establishes} an upper bound on the rate of decay of a Dirichlet function for {intrinsic approximation on} \emph{any} $k$-dimensional {nondegenerate} submanifold $M \subset \R^d$. 
{This is done by exhibiting for every $k \le d$ an explicit constant $c = c(k,d)$ such that for any $M$ as above, the set $\BA_M(\psi_c)$ has full Hausdorff dimension. }

The constant $c(k,d)$ is arrived at via combinatorial considerations and to the authors' knowledge has not appeared previously. In some sense, it represents the heart of the paper. Given the natural way it arises in a volume computation (cf. the proof of Claim \ref{claimNkd}), we suspect that it will play a significant role in intrinsic Diophantine approximation moving forward.
{Here is how it is defined:}

\begin{notation}
\label{notationkd}
Denote
$
[n,m] := \binom{n + m}{m} = \binom{m + n}{n}
$, and for 
% any 
$1\leq k\leq d$, let $n_{k,d}\in\N$ be maximal such that
\begin{equation}
\label{mnkddef}
d = [k-1,1] + [k - 1,2] + \ldots + [k - 1,n_{k,d}] + m_{k,d}
\end{equation}
for some $m_{k,d}\geq 0$, and let $m_{k,d}$ be the unique integer satisfying \eqref{mnkddef}. Let
\[
N_{k,d} :=1[ k-1,1] + 2[k - 1,2] + \ldots + n_{k,d}[k - 1,n_{k,d}] + (n_{k,d} + 1)m_{k,d},
\]
and define $c(k,d) := (d + 1)/N_{k,d}$.
\end{notation}

{ We can now state our main theorem:}

%Note the related problem of establishing a lower bound on the rate of decay is impossible, as evidenced by the optimal Dirichlet functions $\psi_{2/n}$ seen in Example \ref{exampletrivial}. Let us also observe that no Dirichlet function could decay faster than $\psi_{1 + 1/d}$, the optimal Dirichlet function for ambient approximation in $\R^d$. Any nondegenerate submanifold which admitted such a Dirichlet function would contradict the above discussion of extremality. In the following theorem, we improve this bound by showing that for \emph{any} nondegenerate submanifold, any Dirichlet function has a significantly lower rate of decay. This theorem can be interpreted as giving an upper bound on the ``intrinsic exponent of irrationality'' of a manifold $M$, defined as the supremum of $c\geq 0$ such that $\psi_c$ is a Dirichlet function for intrinsic approximation on $M$.

\begin{theorem}
\label{theoremmain}
Let $M \subset \R^d$ be a {nondegenerate} submanifold of dimension $k$. Then $$\dim\big(\BA_M(\psi_{c(k,d)})\big) = k.$$ 
{Consequently,}
%Let $M \subset \R^d$ be a manifold of dimension $k$, with $1 \leq k \leq d$. Then 
no Dirichlet function {for intrinsic approximation on} $M$ decays faster than $\psi_{c(k,d)}$. %where $c(k,d)\leq 1 + 1/d$ is as in Notation \ref{notationkd} below, and $c(k,d) \leq 1$ whenever $k < d$.
\end{theorem}

{As a way to interpret this theorem, we can consider the \emph{intrinsic Dirichlet exponent} of $M$, that is,
$$
\delta^{\rm int}_M := \delta(M,\Q^d\cap M,H).$$
The above theorem implies that $\delta^{\rm int}_M \le c(k,d)$ if $M$ is a nondegenerate submanifold of $\R^d$ of dimension $k$.}

%At a first glance Theorem \ref{theoremmain} may seem somewhat intimidating, since its substantive content depends on the precise value of $c(k,d)$. 

\smallskip

The following remarks may help shed some light on the %significance of this 
constant {$c(k,d)$}: 
%\commkeith{Cleaned up this discussion}:

\begin{itemize}
%\item Although the theorem as stated does not cover the case of $k=d$, our formula in Notation \ref{notationkd} does, and yields that $c(d,d) = 1 + \frac1{d}$. Therefore, considering approximation of $\R^d$ by $\Q^d$ as a special case of intrinsic approximation, the function $\psi_{c(d,d)} = \psi_{1 + 1/d}$ is precisely the optimal Dirichlet function furnished by Dirichlet's theorem \eqref{equationdir}.
\item %The reader can check that only w
{When $k = d$, one has $n_{k,d} = 1$ and $m_{k,d} = 0$; thus $N_{k,d} = d$ and %therefore \commkeith{remove ``therefore"'}} %does 
$c(k,d) = 1 + 1/d$. %\new{Thus\commkeith{Replace ``Thus"' with ``
Therefore the $k = d$ case of Theorem \ref{theoremmain} coincides with Schmidt's theorem, that is, with full Hausdorff dimension of $\BA_d$.}
%\commkeith{I don't think we need to do the intermediary steps of $N_{k,d}$ and $m_{k,d}$} \commkeith{This bullet should add: ``This is seen in the Table below along the diagonal.''}

\item {When $k = d-1$, one has $n_{k,d} = 1$ and $m_{k,d} = 1$; thus $N_{k,d} = d+1$ and %therefore %does 
$c(k,d) = 1$. In particular, this gives a different proof of \cite[Theorem 4.5]{FKMS2}. The latter theorem establishes full Hausdorff dimension of $\BA_M(\psi_1)$, and hence the optimality of Dirichlet function $\psi_1$, for any nonsingular rational quadric hypersurface $M$. Our result extends this to an arbitrary nondegenerate hypersurface $M$, in particular showing that whenever $\psi_1$ is a Dirichlet function for intrinsic approximation on $M$, it must be optimal.} %\commkeith{This is seen in the Table because the super diagonal is always 1.}

\item The reader can check that {if $d$ is fixed, then $c(k,d)$ is strictly increasing with respect to $k$. This confirms the intuitive logic that %since 
higher-dimensional manifolds %, being larger than lower-dimensional ones, 
may have more intrinsic rationals and therefore points on these manifolds should be expected to be better approximable by intrinsic rationals, so their intrinsic Dirichlet exponent %of irrationality 
should be higher. This also implies that for $k < d$, we have $c(k,d) \leq 1$. Therefore for proper nondegenerate submanifolds $M\subset \R^d$, the bound on $\delta^{\rm int}_M$ given by Theorem \ref{theoremmain} is strictly stronger than the ``trivial" bound $$\delta^{\rm int}_M \le \delta^{\rm amb}_M = 1 + 1/d$$
%\commkeith{That inequality should be strict} 
given by considering all rational points, not just intrinsic ones, and using the aforementioned result of Beresnevich on existence of badly approximable vectors on nondegenerate manifolds. }

%onsequences of extremality previously discussed.
\item 
%When the manifold $M$ is either of dimension one or codimension one, then the combinatorial calculation defining $c(k,d)$ becomes particularly simple: we have 
It is also easy to compute that when $k = 1$, one has $n_{k,d} = d$ and $m_{k,d} = 0$; thus $N_{k,d} = d(d+1)/2$ and 
$c(1,d) = 2/d$. For other values of $k,d$ the computation is more involved. Below is a table of values of $c(k,d)$ for $k,d\le 6$, with rows corresponding to $k$ and columns to $d$:
%\commkeith{I am alright with the table, but we should contextualize it. Each of the bullets says something about the table; what they say should be made explicit. And the top row should be written as $2/d$, not in reduced form, with the exception of $1$. This bullet should add the top row is always $2/d$.}
 % and $c(d - 1,d) = 1$ for all $d\geq 2$. In particular, since $c(1,2) = 1$, Theorem \ref{theoremmain} shows that no nondegenerate curve in $\R^2$ can have a Dirichlet function which decays more quickly than $\psi_{2/2} = \psi_1$, as claimed above.
%\item The formula for $c(k,d)$ given in Notation \ref{notationkd} shows that if $d$ is fixed, then $c(k,d)$ is strictly increasing with respect to $k$. This confirms the intuitive logic that since higher-dimensional manifolds are larger than lower-dimensional ones, they may have more intrinsic rationals and therefore points on these manifolds should be expected to be better approximable by intrinsic rationals, so the intrinsic exponent of irrationality should be higher.
\end{itemize}

\begin{center}
\begin{tabular}{|c|cccccc|}
\hline
 & 1 & 2 & 3 & 4 & 5 & 6\\
\hline
1 & 2 & 1 & 2/3 & 1/2 & 2/5 & 1/3\\
2 && 3/2 & 1 & 5/6 & 3/4 & 7/11\\
3 &&& 4/3 & 1 & 6/7 & 7/9\\
4 &&&& 5/4 & 1 & 7/8\\
5 &&&&& 6/5 & 1\\
6 &&&&&& 7/6\\
\hline
\end{tabular}
\end{center}

\ignore{
Although the theorem as stated does not cover the case of $k=d$, our formula in Notation \ref{notationkd} does, and yields that $c(d,d) = 1 + \frac1{d}$. Therefore, considering approximation of $\R^d$ by $\Q^d$ as a special case of intrinsic approximation, the function $\psi_{c(d,d)} = \psi_{1 + 1/d}$ is precisely the optimal Dirichlet function furnished by Dirichlet's theorem \eqref{equationdir}. Furthermore, the computation yields $c(1,d) = 2/d$, and thus taking $d=2$, we see that no nondegenerate curve in $\R^2$ can have a Dirichlet function which decays faster than $\psi_{2/2} = \psi_1$, answering our previous question. Moreover, the formula for $c(k,d)$ given in Notation \ref{notationkd} shows that the greater the gap between the submanifold and the ambient space, the tighter the control on the rate of decay of a Dirichlet function.
}

\ignore{
Two important cases of the computation of $c(k,d)$ (see Notation \ref{notationkd}) are that $$c(d,d) = 1 + \frac{1}{d} \quad \text{and} \quad c(1,d) = \frac{2}{d}.$$ In particular, treating the case of $M = \R^d$ and $\QQ = \Q^d$ as a special case of intrinsic approximation, we see that $\psi_{c(d,d)} = \psi_{1 + 1/d}$ is precisely the optimal Dirichlet function furnished by Dirichlet's theorem \ref{theoremdir}. Furthermore, letting $d=2$ in the second computation, we see that the optimal Dirichlet function for any nondegenerate curve in $\R^2$ cannot decay faster than $\psi_{2/2} = \psi_1$, answering our previous question.

Recently, the authors \cite{FKMS2} studied intrinsic approximation in the setting of quadric hypersurfaces. There it is shown (Theorem 5.1) that the function $\psi_1$ is an optimal Dirichlet function with respect to approximation by rationals lying on the hypersurface. Since in that case $k=d-1$, we see that once again our theorem yields $\psi_{c(d-1,d)}= \psi_1$. 
}

\ignore{To show that no Dirichlet function can decay faster than $\psi_{c(k,d)}$, it will suffice to demonstrate the existence of points which are badly approximable with respect to the function $\psi_{c(k,d)}$. Recall that in our setup, the set of approximants is $\QQ := \Q^d \cap M$.
\begin{definition}
\label{definitionba} We say that $\xx \in M$ is $\psi$-\emph{intrinsically badly approximable}, written $\xx \in \BA_M(\psi)$, if there exists a constant $\epsilon(\xx) > 0$ such that for every $\rr \in \Q^d \cap M$, $\dist(\xx,\rr) \geq \epsilon(\xx) \psi(H(\rr))$. 
\end{definition}
Note that we require $\xx \in M$, and that in the particular case of $M = \R^d$ and $\psi = \psi_{1 + 1/d}$, this reduces to the usual definition of $\BA_d$.}

{Note that %n many cases, including the dimension and codimension one cases mentioned above, 
for some values of $k,d$
%there exists 
we can construct a $k$-dimensional nondegenerate submanifold $M$ of $\R^d$} such that $\psi_{c(k,d)}$ is a Dirichlet function for intrinsic approximation on $M$; {for those} cases, Theorem \ref{theoremmain} demonstrates the optimality of this Dirichlet function %as well as being optimal itself in the sense 
{and shows} %that 
the constant $c(k,d)$ to be best possible. %We will say more about this below.
This is formalized in the following definition:

\begin{definition}
\label{definitionmaximallyapproximable}
We will call a nondegenerate submanifold $M \subset \R^d$ of dimension $k$ \emph{maximally approximable} if $\psi_{c(k,d)}$ is a Dirichlet (and hence an optimal Dirichlet) function.
\end{definition}
Example \ref{exampletrivial} shows that when $n > 2$, the curve $C_n$ is \emph{not} a maximally approximable submanifold of $\R^2$. So we know that $\psi_{c(k,d)}$ is not a Dirichlet function for some submanifolds. Nonetheless, our theorem immediately suggests the following question:
%\begin{question}
{for which} $1 \leq k \leq d$ does there exist a {$k$-dimensional} maximally approximable submanifold %$M \subset \R^d$ 
{of $\R^d$}? %of dimension $k$? 
%\end{question}
%\comkeith{Are we upgrading this to a conjecture?}
We will collect examples and give some partial %positive results 
{answers} to this question in Section \ref{sectionexamples}. 

\smallskip

We will prove Theorem \ref{theoremmain} by showing that $\BA_M(\psi_{c(k,d)})$ is 
%nonempty, and in fact has full Hausdorff dimension as a subset of $M$. 
a winning set of a certain game. Recall that in order to prove {the special case $M = \R^d$ of Theorem \ref{theoremmain},} 
%Before stating our theorem, let us briefly recall the analogous result in the classical case $(\R^d,\Q^d,H)$. That the set of badly approximable points is nonempty was shown in dimension one by Liouville, and in higher dimensions by Perron, who constructed countably infinitely many badly approximable points in every dimension. The full dimensionality of $\BA_d$ in the case of $d=1$ was established by Jarn\'ik \cite{Jarnik1}, and for general dimension $d$ by Schmidt \cite{Schmidt2}. To achieve this, 
Schmidt developed a powerful tool now known as Schmidt's game, a two-player game whose winning sets enjoy many remarkable properties, including having full dimension. In recent years, many variants of the game have been introduced, of particular note the \emph{absolute}
% winning} variant 
{game %introduced by 
of} McMullen \cite{McMullen_absolute_winning}, and the \emph{hyperplane} absolute 
%winning variant 
{game} introduced in \cite{BFKRW} (see Section \ref{sectiongames} below for more details). It is this last variant which we will %be interested in 
{utilize} here. Note that a key ingredient in Schmidt's proof %that %the classical set of badly approximable points 
%$\BA_d$ is winning 
is the Simplex Lemma, whose original statement and proof go back to Davenport and Schmidt \cite[p.\ 57]{Schmidt3}. A crucial step in our proof, and one which we believe to be of independent interest, is establishing an analogue of the Simplex Lemma for rationals constrained to lie on a fixed nondegenerate manifold of $\R^d$ (Lemma \ref{lemmasimplex}). We also develop new tools for utilizing the hyperplane absolute game, which enables us to show that $\BA_M(\psi_{c(k,d)})$ is hyperplane absolute winning. 
%In $\R^d$, this remarkable lemma states that the set of rational points in any ball $B(\xx,\rho)$ whose denominators are less than $\epsilon \rho^{-d/(d+1)}$ is contained in an affine hyperplane, where $\epsilon > 0$ is small and depends only on the dimension $d$. 

\smallskip

{Now let $\lambda_M$ be a smooth volume measure on $M$. It is worthwhile to point out that the conclusion of Theorem \ref{theoremmain}, that is, full Hausdorff dimension of the set $\BA_M(\psi_{c(k,d)})$, cannot in general be upgraded to positive measure. Indeed, it follows from Khintchine's theorem \cite[Theorem III.3A]{Schmidt3} that the Lebesgue measure of the set $\BA_d$ is zero. And a similar result for $\BA_M(\psi_1)$ where $M$ is a nonsingular rational quadric hypersurface is a special case of \cite[Theorem 6.2]{FKMS2}.}

{However, the situation is different if the exponent $c(k,d)$ gets replaced with a slightly bigger one. For an arbitrary Diophantine triple $(M,\QQ,H)$ and $c > 0$ let us introduce the set
\eq{defvwa}{
\begin{split}
&\VWA(M,\QQ,H,\psi_c)
:=M\smallsetminus \bigcup_{\epsilon > 0}\,\BA(M,\QQ,H,\psi_{c+\epsilon})\\
= & \left\{ \xx\in M :\ \begin{aligned} \exists\,\epsilon > 0\text{ and a sequence }\rr_n \in \QQ\text{ such that } \rr_n \to \xx \\ \text{ and }\dist(\xx,\rr_n) \leq \psi_{c+\epsilon}\big(H(\rr_n)\big)
%\psi_\epsilon\big(H(\rr_n)\big) 
\text{ for all $n$}\qquad\end{aligned}\right\}
\end{split}
}
of \emph{$\psi_c$-very well approximable} points.}
% \commkeith{I find this level of generality off-putting to the reader; why are we not just doing this for exponents, since that's all we comment on? It takes some thought to understand what $\psi \psi_\epsilon$ means}
{The fact that the set 
$$
\VWA_d := \VWA(\R^d,\Q^d,H,\psi_{1 + 1/d})
$$
of very well approximable vectors in $\R^d$ is Lebesgue null is an easy consequence of the Borel--Cantelli Lemma. A similar statement for ambient approximation -- namely, that $\lambda_M(\VWA_d) = 0$, where $M\subset \R^d$ is a nondegenerate submanifold -- is much tricker. It has been conjectured by Sprind\v uk in 1980 \cite[Conjecture $H_1$]{Sprindzuk3} and demonstrated by Margulis and the second-named author in 1998 \cite{KleinbockMargulis2}. Later it was shown that $\VWA_d$ is $\mu$-null for other interesting measures $\mu$ on $\R^d$. In particular, it follows from \cite[Theorem 1.1 and Proposition 7.3]{KLW} that $\mu(\VWA_d) = 0$ whenever $\mu$ is an \emph{absolutely friendly} measure (see Definition \ref{definitionAF}) on a submanifold $M$ as above. }

{Our second main theorem gives an intrinsic approximation analogue of the above statement:}

\begin{theorem}[Restated as Theorem \ref{theoremextremal}]
\label{theoremvwanullset}
Let $M \subset \R^d$ be a submanifold of dimension $k$. If $\lambda_M$-almost every point of $M$ is nondegenerate, then 
%the complement of $\BA_M(\psi_{c})$ is a Lebesgue nullset.
\eq{defvwam}
{{\VWA_M(\psi_{c(k,d)}) := \VWA(M,\Q^d\cap M,H,\psi_{c(k,d)})}}
 is a $\lambda_M$-nullset. More generally, let $\Psi:U\to M$ be a local parameterization of $M$, let $\mu$ be an absolutely friendly measure on $U$, and let $\nu = \Psi[\mu]$. If $\nu$-almost every point of $M$ is nondegenerate, then $\VWA_M(\psi_{c(k,d)})$ is a $\nu$-nullset.
\end{theorem}\

\ignore{We will prove Theorem \ref{theoremmain} by showing that $\BA_M(\psi_{c(k,d)})$ is nonempty, and in fact has full Hausdorff dimension as a subset of $M$. Before stating our theorem, let us briefly recall the analogous result in the classical case $(\R^d,\Q^d,H)$. That the set of badly approximable points is nonempty was shown in dimension one by Liouville, and in higher dimensions by Perron, who constructed countably infinitely many badly approximable points in every dimension. The full dimensionality of $\BA_d$ in the case of $d=1$ was established by Jarn\'ik \cite{Jarnik1}, and for general dimension $d$ by Schmidt \cite{Schmidt2}. To achieve this, Schmidt developed a powerful tool now known as Schmidt's game, a two-player game whose winning sets enjoy many remarkable properties, including having full dimension. In recent years, many variants of the game have been introduced, of particular note the \emph{absolute winning} variant introduced by McMullen \cite{McMullen_absolute_winning}, and the \emph{hyperplane} absolute winning variant introduced in \cite{BFKRW} (see Section \ref{sectiongames} below for more details). It is this last variant which we will be interested in here. A key ingredient in the proof that %the classical set of badly approximable points 
$\BA_d$ is winning is the Simplex Lemma, whose original statement and proof go back to Davenport and Schmidt \cite[p.57]{Schmidt3}. In $\R^d$, this remarkable lemma states that the set of rational points in any ball $B(\xx,\rho)$ whose denominators are less than $\epsilon \rho^{-d/(d+1)}$ is contained in an affine hyperplane, where $\epsilon > 0$ is small and depends only on the dimension $d$. 

A crucial step in our proof, and one which we believe to be of independent interest, is establishing an analogue of the Simplex Lemma for rationals constrained to lie on a fixed nondegenerate manifold of $\R^d$ (Lemma \ref{lemmasimplex}). We also develop new tools for utilizing the hyperplane absolute game, which enables us to show that $\BA_M(\psi_{c(k,d)})$ is hyperplane absolute winning. Since hyperplane absolute winning sets have full dimension, we obtain
\begin{theorem}
\label{theorembafulldim}
Let $M \subset \R^d$ be a nondegenerate submanifold of dimension $k$. Then $$\dim(\BA_M(\psi_{c(k,d)})) = k.$$ 
\end{theorem}

As noted previously, this implies Theorem \ref{theoremmain}. 
\smallskip

Before stating our last result, let us briefly return to the classical case $(\R^d,\Q^d,H)$ once more. As remarked previously, the set $\VWA_d$ of points which are $\psi_c$-approximable for some $c > 1 + 1/d$ is a Lebesgue null set. If, for fixed $c > 1 + 1/d$, we define $\VWA_d(\psi_c) := \bigcup_{c' > c} A(\psi_{c'})$, then the Hausdorff dimension of $\VWA_d(\psi_c)$ is $\frac{d+1}{c}$. Our next goal is to derive analogous results in the intrinsic setting. 

\begin{definition}
We say that $\xx \in M$ is \emph{intrinsically} $\psi$-\emph{approximable}, written $\xx \in A_M(\psi)$, if there exists a sequence $\rr_n \in \Q^d \cap M$ such that $\rr_n \to \xx$ and $\dist(\xx,\rr_n) \leq \psi(H(\rr_n))$.

Given $c > 0$, we say that $\xx \in M$ is \emph{intrinsically $\psi_c$-very well approximable}, written $\xx \in \VWA_M(\psi_c)$, if there exists $c' > c$ such that $\xx$ is $\psi_{c'}$-approximable, where $\psi_c$, $\psi_{c'}$ are as in \eqref{psi_c}. In other words, we set $$\VWA_M(\psi_c) := \bigcup_{c' > c} A_M(\psi_{c'}).$$ 
\end{definition}

Another consequence of the simplex lemma is that the Lebesgue measure $\lambda_M$ of $\VWA_M(\psi_{c(k,d)})$ is null; in fact this result holds for a much larger class of measures, namely the class of \emph{absolutely friendly} measures (see Definition \ref{definitionAF}). 

\begin{theorem}[Restated as Theorem \ref{theoremextremal}]
\label{theoremvwanullset}
Let $M \subset \R^d$ be a submanifold of dimension $k$. If $\lambda_M$-almost every point of $M$ is nondegenerate, then $\VWA_M(\psi_{c(k,d)})$ is a Lebesgue nullset.

More generally, let $\Psi:U\to M$ be a local parameterization of $M$, let $\mu$ be an absolutely friendly measure on $U$, and let $\nu = \Psi[\mu]$. If $\nu$-almost every point of $M$ is nondegenerate, then $\VWA_M(\psi_{c(k,d)})$ is a $\nu$-nullset.
\end{theorem}
%\comkeith{Got rid of the reference to $D$-nondegenerate. Do we need it?}

This notation is motivated by the classical notion of $\VWA_d$. If we know a manifold $M \subset \R^d$ is maximally approximable as discussed above, then it makes sense to call $\VWA_M(\psi_{c(k,d)})$ the set of \emph{intrinsically very well approximable} points on $M$. Trivially this is satisfied for $M = \R^d$, as well as the other examples of maximally approximable submanifolds discussed below in Section \ref{sectionexamples}.}

\begin{remark}
Theorem \ref{theoremvwanullset} immediately implies a result about \emph{extrinsic} approximation, {i.e.}\ the approximation of points on a manifold $M$ by rational points in the complement of $M$, described by the Diophantine triple $(M,\Q^d\smallsetminus M,H)$. Namely, since the exponent $c(k,d)$ is strictly less than the exponent $1 + 1/d$ appearing in Dirichlet's theorem, Theorem \ref{theoremvwanullset} implies that for almost all $\xx\in M$, only finitely many of the approximants from Dirichlet's theorem can lie inside $M$, so $\xx$ is extrinsically $\psi_{1 + 1/d}$-approximable. For further discussion of extrinsic approximation see \cite{FishmanSimmons2}, where a result is proven for every point in $M$ (not just almost every point) which cannot be deduced from a corresponding ``intrinsic badly approximable'' result.
\end{remark}

\begin{question}
An interesting question is whether Theorem \ref{theoremvwanullset} can be strengthened by estimating the Hausdorff dimension of $\VWA_M(\psi_c)$ for a fixed $c > c(k,d)$. We do this for nonsingular quadric hypersurfaces in \cite{FKMS2} (see \cite{FMS} for a generalization), where the structure of quadric hypersurfaces is explicitly used. 
\end{question}

{\bf Acknowledgements.} The first-named author was supported in part by the Simons Foundation grant \#245708. The second-named author was supported in part by the NSF grant DMS-1101320. The fourth-named author was supported in part by the EPSRC Programme Grant EP/J018260/1. The authors would also like to thank Victor Beresnevich, Yann Bugeaud, and Sanju Velani for helpful discussions. The authors thank the anonymous referee for valuable comments.
\smallskip

{\bf Outline.} In Section \ref{sectionexamples}, we discuss some ways of constructing maximally approximable manifolds. In Section \ref{sectiongames}, we recall the definition of the hyperplane absolute game and introduce two variants, which turn out to be equivalent to the original game. In Section \ref{sectionsimplex}, we prove the main lemma of this paper, an intrinsic analogue of the Simplex Lemma, and use it to prove our main Theorems \ref{theoremmain}
%{theorembafulldim} 
and \ref{theoremvwanullset}.

\begin{convention}
Throughout the paper, the symbols $\lesssim_\times$, $\gtrsim_\times$, and $\asymp_\times$ will denote multiplicative asymptotics. For example, $A\lesssim_{\times,K} B$ means that there exists a constant $C > 0$ (the \emph{implied constant}), depending only on $K$, such that $A\leq C B$. In general, dependence of the implied constant on universal objects such as the manifold $M$ will be omitted from the notation.
\end{convention}

\begin{convention}
The symbol $\triangleleft$ will be used to indicate the end of a nested proof.
\end{convention}

\ignore{
\section{Brief Interlude on Projective Space}
\label{sectionprojective}

 We now consider the main setup of the paper, namely that of intrinsic approximation. One way to do it is to take $X = \R^d$, choose a $k$-dimensional submanifold $M$ of $\R^d$, let $\QQ = \Q^d\cap M$ and set $H(\frac{\pp}{q}) = q$ as in the introduction. However, we have chosen a different approach: state and prove the main results of the paper for submanifolds of projective spaces. This way in most cases statements of results and their proofs become more natural and transparent, see Remark \ref{remarkprojective} below. 

Let $\R^d$ denote the $d$-dimensional real projective space, and let $\pi:\R^d\butnot\{\0\}\to \R^d$ be the quotient map $\pi(\xx) := \xx$. For a subset $S$ of $\R^d$, we let $[S] = \pi(S\butnot\{\0\})$. With some abuse of notation, let us define the \emph{standard height function} $H:\Q^d\to\N$ by the formula \eq{projective height}{\begin{aligned}H([\pp]) = \|\pp\|, \text{ where $\pp$ is the unique (up to a sign)}\\ \text{ primitive integer representative of }[\pp].\qquad \end{aligned}}
Here and elsewhere $\|\cdot\|$ represents the max norm.

\begin{remark}
\label{remarkprojective}
%Many of the results stated below can be reformulated as results about submanifolds in $\R^d$ with no significant loss of information. However, others cannot, and so we have stated them in projective space in order to preserve this key information. 
To understand the difference between results for affine and projective spaces, note that if $\iota_d:\R^d\to\R^d$ is given by the formula $\iota_d(\xx) = [1:\xx]$ and if $B\subset\R^d$ is a bounded set, then $\iota_d\given B$ is bi-Lipschitz and
\begin{equation}
\label{Hstdcomparison}
H\big(\iota_d(\rr)\big) \asymp_{\times, B} H(\rr) \all \rr\in\Q^d\cap B.
\end{equation}

In particular, if $M$ is a nondegenerate submanifold of $\R^d$, then $\iota_d(M)$ is a nondegenerate\Footnote{Cf. Definition \ref{definitionnondegenerateprojective}.} submanifold of $\R^d$, and the Diophantine triples $T_{\aff} := (M,\Q^d\cap M,H)$ and $T_{\mathrm{proj}} := \big(\iota_d(M),\Q^d\cap\iota_d(M),H\big)$ are ``locally isomorphic''. However, both the bi-Lipschitz constant and the implied constant of \eqref{Hstdcomparison} depend on the chosen bounded set $B$. Thus concepts which are robust under point-dependent multiplicative constants will not be affected by the transformation. For example, whether or not a function is Dirichlet will be the same for the triples $T_{\aff}$ and $T_{\mathrm{proj}}$, but it is conceivable that a function could be uniformly Dirichlet for the triple $T_{\mathrm{proj}}$ but not for the triple $T_{\aff}$.(Eliminate use of triples).

One advantage of using the projective setup is that the projective statements are sometimes stronger than the affine corollaries, for example Dirichlet's proof of \eqref{equationdir} actually shows that 
\begin{equation}
\label{eqdirproj}
\psi_{1 + 1/d} \text{ is a uniform Dirichlet function for } \P^d_\Q \subset \P^d_\R.
\end{equation} 
This result is stronger than \eqref{equationdir} in that it implies \eqref{equationdir} but simply translating \eqref{equationdir} to projective space as above does not yield \eqref{eqdirproj}.
\end{remark}

 Additionally, the use of the projective setting makes the paper compatible with a recent study of intrinsic approximation on quadric hypersurfaces \cite{FKMS2}. 

In what follows we will work with submanifolds $M \subset \P^d_\R$, approximating points by rationals in $\QQ := \P^d_\Q \cap M$, with respect to the height function $H$. Given a nonincreasing function $\psi: \N \to (0,\infty)$, we say that $\xx \in M$ is $\psi$-\emph{approximable} if there exists a sequence of rationals $[\rr_n] \in \QQ$ satisfying $$[\rr_n] \to \xx
\quad \text{ and } \quad \dist(\xx, [\rr_n]) \leq \psi(H([\rr_n])).$$We denote the set of $\psi$-approximable points in $M$ by $A_M(\psi)$. We will say that $\xx \in M$ is $\psi$-\emph{badly approximable} if there exists a constant $\epsilon(\xx)$ such that for every $[\rr_n] \in \QQ$, $\dist(\xx, [\rr_n]) \geq \epsilon(\xx) \psi(H([\rr_n]))$. We denote the set of $\psi$-badly approximable points in $M$ by $\BA_M(\psi)$. Finally, we define $\psi_c(t) = \frac1{t^c}$. (Clean up these definitions)
}

\section{A discussion of maximal approximability in special cases}
\label{sectionexamples}

%We recall that if $M\subset\R^d$ is a manifold, a \emph{local parameterization} is a diffeomorphism from an open subset of $\R^d$ to an open subset of $M$.

We start by giving a more detailed definition of a nondegenerate submanifold of $\R^d$.

\begin{definition}[Cf.\ {\cite[p.\ 341]{KleinbockMargulis2}}]
\label{definitionnondegenerate} 
Let $M\subset\R^d$ be a submanifold of dimension $k$. For each $\xx\in M$ and $j\in\Namer$, let
\[
T_{\xx}^{(j)}(M) := \bigcup_{i = 1}^j \{\gamma^{(i)}(0)\upharpoonleft \gamma:(-\epsilon,\epsilon)\rightarrow M,\;\gamma(0) = \xx\} \subset \R^d;
\]
equivalently, if $\emb:U\to M$ is a coordinate chart satisfying $\emb(\vv) = \xx$,
\begin{equation}
\label{TxjM2}
T_{\xx}^{(j)}(M) = \Span\big(\big\{\del^\alpha\emb(\vv) : \alpha\in\Neur^k , 0 < |\alpha| \leq j \big\}\big).
\end{equation}
Here $\Neur := \N\cup\{0\}$, and the power $ \del^\alpha$ is taken using multi-index notation: if $\alpha = (\alpha_1,\dots,\alpha_k)$, then $\del^\alpha = \del_1^{\alpha_1}\cdots \del_k^{\alpha_k}$. We will call $T_{\xx}^{(j)}(M)$ the \emph{tangent space of order $j$ of $M$ at $\xx$}.

A point $\xx\in M$ is said to be \emph{$D$-nondegenerate} for $M$ if $T_{\xx}^{(D)}(M) = \R^d$, and \emph{nondegenerate} if it is nondegenerate for some $D\in\Namer$. Finally, $M$ is \emph{nondegenerate} if some point of $M$ is nondegenerate for $M$.
\end{definition}

\begin{observation} \label{observationdegeneracy}
~
\begin{itemize}
\item[(i)] If $M$ is contained in an affine hyperplane of $\R^d$, then $M$ is degenerate at every point. 
\item[(ii)] If $M$ is real-analytic and connected, and $M$ is not contained in an affine hyperplane, then every point of $M$ is nondegenerate.
\item[(iii)] For each $D\in\Namer$, the set of $D$-nondegenerate points of $M$ is relatively open in $M$.
\end{itemize}
\end{observation}
\begin{proof}
~
\begin{itemize}
\item[(i)] If $M\subset \LL + \vv$ where $\LL$ is a linear hyperplane and $\vv\in\R^d$, then $T_{\xx}^{(D)}(M)\subset \LL$ for all $\xx\in M$ and $D\in\Namer$.
\item[(ii)] If $M$ is real-analytic and degenerate at $\xx\in M$, then $M\subset \xx + \bigcup_{D\in\Namer}T_{\xx}^{(D)}(M)$.
\item[(iii)] This follows from \eqref{TxjM2} together with the lower semicontinuity of the function sending a matrix to its rank.
\qedhere\end{itemize}
%If $M$ is contained in a hyperplane $\LL$, then $M$ is degenerate at every point, since for every $\xx\in M$, $T_{\xx}^{(D)}(M)$ is contained in the linear part of $\LL$. The converse is true under the additional assumption that $M$ is real-analytic, since if $M$ is degenerate at some point $\xx\in M$, then $M\subset \xx + \bigcup_{D\in\N} T_{\xx}^{(D)}(M)$. Moreover, if $T_{\xx}^{(D)}(M) = \R^d$ for \emph{some} $\xx\in M$, then there is a neighborhood $B$ of $\xx$ such that $M\cap B$ is nondegenerate.
\end{proof}
Connected manifolds which are degenerate at every point but are not contained in a hyperplane exist but are very pathological; we refer to \cite{Wolsson} for a detailed account, stated in somewhat different language.

\smallskip

\ignore{The next definition is, in some sense, the heart of the paper. \commkeith{We should reword this then} We introduce it here for the sake of exposition, but its motivation will come later, specifically Claim \ref{claimNkd}. There we see that the terms $N_{k,d}$ and $c(k,d)$ arise naturally in a volume computation used to prove the intrinsic Simplex Lemma.}
The next example describes an important family of nondegenerate submanifolds of $\R^d$:

\begin{example}[Veronese variety]
\label{exampleveronese}
Fix $k,n\in\Namer$, let $d = [k,n] - 1$ (see Notation \ref{notationkd}), and consider the \emph{Veronese embedding} $\Psi_{k,n}:\R^k\rightarrow\R^d$ defined by
\[
\Psi_{k,n}(\tt) = (\tt^\alpha)_{\substack{\alpha\in\Neur^k \\ 0 < |\alpha| \leq n}},
\]
where the power $\tt^\alpha$ is taken using multi-index notation. Then it can be straightforwardly verified that the \emph{Veronese variety} $V_{k,n} = \Psi_{k,n}(\R^k)$ is a nondegenerate\Footnote{Even stronger, every point of $V_{k,n}$ is $n$-nondegenerate. Moreover, by \eqref{TxjM2}, $\HD\big(T_{\xx}^{(n)}(M)\big) \leq [k,n]$ for any manifold $M$. Thus the ambient dimension of $V_{k,n}$ is maximal among all $n$-nondegenerate $k$-dimensional manifolds.} submanifold of $\R^d$.
\end{example}
The one-dimensional special case ($k = 1$, $d = n$) is usually called \emph{Veronese curve} or \emph{rational normal curve}. Letting $k = 1$ and $n = 2$ yields $V_{1,2} = C_2$ (cf.\ Example \ref{exampletrivial}). Recall that the latter curve was our first example of a maximally approximable manifold (see Definition \ref{definitionmaximallyapproximable}). 

%\comkeith{Remove references to $H_\std$}

The following lemma shows that the map $\Psi_{k,n}$ is an ``isomorphism'' between the Diophantine triples $(\R^k,\Q^k,H^n)$ and $(V_{k,n},V_{k,n}\cap\Q^d,H)$:

\begin{lemma}
\label{lemmaveronese}
The map $\Psi_{k,n}$ is a diffeomorphism between $\R^k$ and $V_{k,n}$; moreover
%\begin{align*}
$$
V_{k,n}\cap \Q^{[k,n] - 1} %&
= \Psi_{k,n}(\Q^k)%\\
\quad\text{and}\quad
H\big(\Psi_{k,n}(\rr)\big) %&
= H^n(\rr) \all \rr\in\Q^k.\footnotemark
$$
%\end{align*}
\end{lemma}
\Footnotetext{The map $\Psi_{k,n}$ is not the only embedding which is an isomorphism in this sense; more generally, if $\Psi:\R\to \R^d$ is an embedding defined by polynomials with integer coefficients, then a relation between $H\circ\Psi$ and $H$ was found in \cite[Proof of Lemma 2]{BDL}. Similarly to Corollary \ref{corollaryveronese}, this relation can be used to discover an optimal Dirichlet function on the corresponding curve. However, in most cases the resulting curve is not maximally approximable.}

The proof is a straightforward computation which is left to the reader.

\begin{corollary}
\label{corollaryveronese}
For any $k,n\in\Namer$, $V_{k,n}$ is a maximally approximable submanifold of $\R^d$.
\end{corollary}
\begin{proof}
We begin by proving the following more general assertion:
\begin{lemma}
\label{lemmaquadraticveronese}
Fix $d,n\in\Namer$, and let $M$ be a maximally approximable submanifold of $\R^d$ of dimension $k$. Suppose that $\Psi_{d,n}(M)$ is a nondegenerate submanifold of $\R^{[d,n] - 1}$.\Footnote{This follows, for example, if $M$ is connected, real-analytic, and Zariski dense in $\R^d$.} Then $\Psi_{d,n}(M)$ is maximally approximable if and only if
\begin{equation}
\label{quadraticveronese}
\frac1n \frac{d + 1}{N_{k,d}} = \frac{[d,n]}{N_{k,[d,n] - 1}}\cdot
\end{equation}
\end{lemma}
\begin{subproof}
Since $M$ is maximally approximable, $\psi_{c(k,d)}$ is an optimal Dirichlet function for the Diophantine triple $(M,M\cap\Q^d,H)$. Now fix $\Psi_{d,n}(\xx)\in \Psi_{d,n}(M)$. Then there exists a sequence $M\cap \Q^d\ni\rr_m\to\xx$ with $\|\rr_m - \xx\| \lesssim_\times \psi_{c(k,d)}\circ H(\rr_m)$, so by Lemma \ref{lemmaveronese}, $\Psi_{d,n}(M)\cap \Q^{[d,n] - 1} \ni \Psi_{d,n}(\rr_m) \to \Psi_{d,n}(\xx)$ and $\|\Psi_{d,n}(\rr_m) - \Psi_{d,n}(\xx)\| \lesssim_\times \psi_{c(k,d)}\circ H(\rr_m) = \psi_{c(k,d)/n)}\circ H(\Psi_{d,n}(\rr_m))$. Since $\xx$ was arbitrary, $\psi_{c(k,d)/n}$ is a Dirichlet function for the Diophantine triple
%It follows that the function $\psi_{c(k,d)/n}$ is an optimal Dirichlet function for the Diophantine triple $(M,M\cap\Q^d,H^n)$, and thus (by Lemma \ref{lemmaveronese}) also for the Diophantine triple
\eq{triple}{\big(\Psi_{d,n}(M),\Psi_{d,n}(M)\cap\Q^{[d,n] - 1},H\big).}
A siimlar argument shows that $\psi_{c(k,d)/n}$ is optimal for \equ{triple}.

On the other hand, $\Psi_{d,n}(M)$ is maximally approximable if and only if $\psi_{c(k,[d,n] - 1)}$ is an optimal Dirichlet function for \equ{triple}. Since $\psi_{c_1}$ and $\psi_{c_2}$ cannot be optimal Dirichlet functions for the same Diophantine triple unless $c_1 = c_2$, the lemma follows.
\end{subproof}

Since $\R^k$ is a maximally approximable submanifold of itself, to complete the proof it suffices to show that \eqref{quadraticveronese} holds when $k = d$. Indeed,\Footnote{In establishing these formulas, the identity $[a,b] = [a - 1,b] + [a,b - 1]$ is useful.}
\begin{align*}
n_{k,k} &= 1
&
m_{k,k} &= 0
&
N_{k,k} &= k
\\
n_{k,[k,n] - 1} &= n
&
m_{k,[k,n] - 1} &= 0
&
N_{k,[k,n] - 1} &= k[k + 1,n - 1] = \frac{kn}{k + 1} [k,n],
\end{align*}
which implies the desired result.
\end{proof}

It will be observed that Corollary \ref{corollaryveronese} is simply the end result of transferring Dirichlet's theorem in $\R^k$ into $V_{k,n}$ via the map $\Psi_{k,n}$. Thus, intrinsic Diophantine approximation on $V_{k,n}$ is essentially the same as Diophantine approximation on $\R^k$, and does not introduce any new phenomena. 

By contrast, new phenomena appear when we study intrinsic approximation on nonsingular quadric hypersurfaces, demonstrating that this theory cannot be reduced to Diophantine approximation on $\R^k$ in the same way. In \cite{FKMS2} we establish a complete theory of intrinsic approximation on quadric hypersurfaces, in particular 
%establishing 
{showing} that $\psi_1$ is a Dirichlet function for every quadric hypersurface \cite[Theorem 5.1]{FKMS2}, regardless of the dimension $k$. Since 
%one can directly check that 
$c(k,k+1) = 1$ for every $k$, this shows that quadric hypersurfaces are maximally approximable.

We end this section with a discussion of the following question: For what pairs $(k,d)$ ($1 \leq k \leq d$) can we prove that there exists a maximally approximable submanifold of $\R^d$ of dimension $k$? For convenience let
\[
\MM = \{(k,d) : \text{ there exists a maximally approximable submanifold of $\R^d$ of dimension $k$}\}.
\]
Trivially $(k,k)\in \MM$ for all $k\in\Namer$. Moreover, since every nonsingular rational quadric hypersurface is maximally approximable, we have $(k,k + 1)\in \MM$ for all $k\in\Namer$. On the other hand, by Corollary \ref{corollaryveronese}, we have $(k,[k,n] - 1)\in\MM$ for all $k,n\in\Namer$. Taking the special case $k = 1$, we have $(1,d)\in\MM$ for all $d\in\Namer$. Thus in every dimension, there exist both a maximally approximable curve and a maximally approximable hypersurface.

It is theoretically possible to get more pairs in $\MM$ by using Lemma \ref{lemmaquadraticveronese}. Namely, if $(k,d)\in\MM$ and if \eqref{quadraticveronese} holds for some $n\in\N$, then $(k,[d,n] - 1)\in\MM$. However, we do not have any examples of pairs $(k,d)$ which we can prove to be in $\MM$ this way but which were not proven to be in $\MM$ in the above paragraph.% For example, if $M$ is a nonsingular rational quadratic hypersurface in $\R^{k + 1}$, then $\Psi_{k + 1,k}(M)$ is a maximally approximable manifold in $\R^{[k + 1,k] - 1}$,\Footnote{To verify $\eqref{quadraticveronese}_{d = k + 1,n = k}$, we observe that the equivalent equation $\eqref{quadraticveronese}_{d = k, n = k + 1}$ was proven in the course of proving Corollary \ref{corollaryveronese}.} but the dimension pair $(k,[k + 1,k] - 1)\in\MM$ was already known from Corollary \ref{corollaryveronese}.
% Example deleted because the resulting manifold is degenerate (since quadratic varieties are not Zariski dense). DS

Although the list of pairs known to be in $\MM$ is so far quite meager, the elegance of the calculation which produces the number $c(k,d)$ (cf.\ Lemma \ref{lemmasimplex} and its proof) leads the authors to believe that there could be many more examples. It is even conceivable that all dimension pairs are in $\MM$. %We therefore ask the following question:

%\begin{openquestion}
%Is it true that for every $1\leq k\leq d$, there exists a maximally approximable submanifold of $\R^d$ of dimension $k$?
%\end{openquestion}

The smallest pair $(k,d)$ for which we do not know the answer to this question is the pair $(2,4)$, which satisfies $c(2,4) = 5/6$.

\section{The hyperplane game and two variants}
\label{sectiongames}

In \cite{Schmidt1}, W.\ M.\ Schmidt introduced the game which is now known as Schmidt's game. A variant of this game was defined by C.\ T.\ McMullen \cite{McMullen_absolute_winning}, and in turn a variant of McMullen's game was defined in \cite{BFKRW}. For the purposes of this paper, we will be interested only in this last variant, called the \emph{hyperplane absolute game},\Footnote{In what follows we abbreviate ``hyperplane absolute game'' to just ``hyperplane game''.} and not in Schmidt's game or McMullen's game. However, we note that every hyperplane winning set is winning for Schmidt's game \cite[Proposition 2.3(a)]{BFKRW}. Some recent papers in which the hyperplane game has appeared are \cite{AGK, KleinbockLy, NesharimSimmons}.

Given $\beta > 0$ and $k\in\N$, the $\beta$-hyperplane game is played on $\R^k$ by two players Alice and Bob as follows:
\begin{itemize}
\item[1.] Bob chooses an initial ball $B_0 = B(\xx_0,\rho_0) \subset \R^k$. (In this paper all balls are closed.)
\item[2.] After Bob's $n$th move $B_n$, Alice chooses an affine hyperplane $A_n\subset\R^k$. We say that Alice ``deletes the neighborhood of $A_n$''.
\item[3.] After Alice's $n$th move $A_n$, Bob chooses a ball $B_{n + 1} = B(\xx_{n + 1},\rho_{n + 1})$ satisfying
\[
B_{n + 1}\subset B_n\butnot A_n^{(\beta\rho_n)} \text{ and }\rho_{n + 1}\geq \beta\rho_n.
\]
Here and elsewhere $S^{(\epsilon)}$ denotes the closed $\epsilon$-thickening of a set $S$. If he is unable to choose such a ball, he loses.\Footnote{This disagrees with the convention introduced in \cite{BFKRW}; however, if we restrict to $0 < \beta \leq 1/3$ (as is done in \cite{BFKRW}) then Bob is always able to make a legal move, so the question is irrelevant. We use the convention that Bob loses in order to avoid technicalities (cf.\ \cite[p.4]{BFS1}) in the variants of the hyperplane game discussed below, where it is not always obvious whether or not Bob has legal moves.}
\end{itemize}
A set $S\subset\R^k$ is said to be \emph{$\beta$-hyperplane winning} if Alice has a strategy which guarantees that
\[
\bigcap_{n = 1}^\infty B_n \cap S \neq\emptyset.
\]
$S$ is \emph{hyperplane winning} if it is $\beta$-hyperplane winning for all $\beta > 0$.

\begin{remark}\label{remarksingleton}
By modifying slightly the proof of \cite[Proposition 4.4]{FSU4}, one can show that if Bob's balls are required to satisfy $\rho_{n + 1} = \beta\rho_n$ rather than $\rho_{n + 1} \geq \beta\rho_n$, then the class of sets which are hyperplane winning remains unchanged. Thus we can assume that $\rho_n\to 0$, in which case the intersection $\bigcap_1^\infty B_n$ is a singleton.
\end{remark}

We list here three important results regarding hyperplane winning sets, the proofs of which can be found in \cite[Proposition 2.3(b,c), Lemma 4.1, and Proposition 4.7]{BFKRW}:

\begin{proposition}
\label{propositionwinningproperties}
~
\item[(i)] The countable intersection of hyperplane winning sets is hyperplane winning.
\item[(ii)] The image of a hyperplane winning set under a $\CC^1$ diffeomorphism of $\R^k$ is hyperplane winning.
\item[(iii)] The intersection of any hyperplane winning set with any open set has full Hausdorff dimension.
\end{proposition}

In fact, in (iii) more is true: the intersection of a hyperplane winning set with a sufficiently nondegenerate fractal (a \emph{hyperplane diffuse} set) is winning for Schmidt's game on that fractal \cite[Propositions 4.7 and 4.9]{BFKRW} and therefore has Hausdorff dimension equal to at least the lower pointwise dimension of any measure whose support is equal to that fractal \cite[Proposition 5.1]{KleinbockWeiss2}. In particular, if the fractal is Ahlfors regular then the intersection has full dimension relative to the fractal.

In \cite[\63]{KleinbockWeiss4}, the notion of hyperplane winning was generalized from subsets of Euclidean space to subsets of arbitrary manifolds. Namely, a subset $S$ of a manifold $M$ is \emph{hyperplane winning relative to $M$} if whenever $\Psi:U\to M$ is a local parameterization of $M$ and $K\subset U$ is compact, the set $\Psi^{-1}(S)\cup(\R^k\butnot K)$ is hyperplane winning.\Footnote{In \cite{KleinbockWeiss4} the definition is given in a slightly different way, depending on the notion of hyperplane winning subsets of an open set. However, it is easily verified that the two definitions are equivalent.}%However, one should be careful to note that a hyperplane winning subset of a manifold $M\subset\R^d$ is not hyperplane winning when interpreted as a subset of $\R^d$; indeed, the set $\R^d\butnot M$ is hyperplane winning on $\R^d$ since $M$ ``looks locally like a hyperplane'' and so can be deleted by Alice (for details see the proof of Theorem \ref{theoremwinningequivalence} below).

We now state our main result concerning the abundance of badly intrinsically approximable points:

\begin{theorem}[Restated as Theorem \ref{theoremBAwinning}]
\label{theoremBAwinningother}
Let $M\subset\R^d$ be a submanifold of dimension $k$, and let $c(k,d)$ be as in Notation \ref{notationkd}. Suppose that for some $D\in\N$, every point of $M$ is $D$-nondegenerate. Then $\BA_M(\psi_{c(k,d)})$ is hyperplane winning relative to $M$.
\end{theorem}

Using Theorem \ref{theoremBAwinningother}, we deduce as a corollary a theorem stated in the introduction:

%\begin{theorem}
%\label{theoremBAfulldimension}
%Let $M\subset\R^d$ be a nondegenerate submanifold of dimension $k$. Then $\HD\big(\BA_M(\psi_{c(k,d)})\big) = k$.
%\end{theorem}
\begin{proof}[Proof of Theorem \ref{theoremmain}
%{theorembafulldim} 
using Theorem \ref{theoremBAwinningother}]
For each $D\in\N$, let $M_D\subset M$ be the set of $D$-nondegenerate points of $M$. Since $M$ is nondegenerate, there exists $D\in\N$ such that $M_D\neq\emptyset$; then $\BA_M(\psi_{c(k,d)})\cap M_D$ is hyperplane winning relative to $M_D$. By (iii) of Proposition \ref{propositionwinningproperties}, $\BA_M(\psi_{c(k,d)})\cap M_D$ has Hausdorff dimension $k$.
\end{proof}

In order to prove Theorem \ref{theoremBAwinningother}, we will introduce two variants of the hyperplane game. The first allows Alice to delete neighborhoods of algebraic sets rather than just hyperplanes, and the second allows her to delete neighborhoods of levelsets of smooth functions. It will turn out that each of these variants is equivalent to the hyperplane game, meaning that any set which is winning for one of the games is winning for all three games.

\begin{definition}
Fix $\beta > 0$ and $D\in\N$. The rules of the \emph{$(\beta,D)$ algebraic-set game} are the same as the rules of the $\beta$-hyperplane game, except that $A_n$ is allowed to be the zero set of any nonzero polynomial of degree at most $D$. A set is \emph{algebraic-set winning} if there exists $D\in\N$ so that it is $(\beta,D)$ algebraic-set winning for all $\beta > 0$.
\end{definition}

Given a ball $B\subset\R^k$ and a $\CC^D$ function $f:B\to\R$, for each $\xx\in B$ let
\[
\|f\|_{\CC^D,\xx} := \max_{\substack{\alpha\in\Neur^k \\ |\alpha|\leq D}} |f^{(\alpha)}(\xx)|\,,
\]
where the derivative is taken using multi-index notation. Let
\[
\|f\|_{\CC^D,B} := \sup_{\xx\in B}\|f\|_{\CC^D,\xx}\,.
\]

\begin{definition}
The rules of the \emph{$(\beta,D,C_1)$-levelset game} are the same as the rules of the $\beta$-hyperplane game, except that $A_n$ is allowed to be the zero set of any nonzero $\CC^{D + 1}$ function $f:B_n\to\R$ satisfying
\begin{equation}
\label{CDbounds}
\|f\|_{\CC^{D + 1},B_n} \leq C_1\|f\|_{\CC^D,B_n}.
\end{equation}
A set is \emph{levelset winning} if there exist $D\in\N$ and $C_1 > 0$ so that it is $(\beta,D,C_1)$-levelset winning for all $\beta > 0$.
\end{definition}

The condition \eqref{CDbounds} should be interpreted heuristically as meaning that ``$f$ is close to being a polynomial of degree $D$''.

Clearly, any hyperplane winning set is algebraic-set winning and any algebraic-set winning set is levelset-winning. The remainder of this section is devoted to the proof of the following theorem:

\begin{theorem}
\label{theoremwinningequivalence}
Any levelset winning set is hyperplane winning.
\end{theorem}

We begin by introducing some notation.
\medskip

\begin{notation}~
\begin{itemize}
\item For $f:U\to\R$, $Z_f$ will denote the zero set of $f$, i.e. $Z_f = f^{-1}(0)$.
\item For $D\in\N$, $\PP_D$ will denote the set of all polynomials of degree at most $D$ whose largest coefficient has magnitude $1$. Note that $\PP_D$ is a compact topological space; moreover, every nonzero polynomial of degree at most $D$ is a scalar multiple of an element of $\PP_D$.
\end{itemize}
\end{notation}

\begin{lemma}
\label{lemmaavoidvariety}
Fix $k\in\N$ and $0 < \beta\leq 1$, and let $f:\R^k\rightarrow\R$ be a nonzero polynomial. Suppose that Bob and Alice are playing the $\beta$-hyperplane game, and suppose that Bob's first move is $B_0 = B(\0,1)$. Then there exists $\gamma > 0$ so that Alice has a strategy to guarantee that Bob's first ball of radius less than $\gamma$ (assuming that such a ball exists) is disjoint from $Z_f^{(\gamma)}$.
\end{lemma}
\begin{proof}
The proof is by induction on the degree of $f$. If $\deg(f) = 0$, then $f$ is a nonzero constant and $Z_f = \emptyset$, so the lemma is trivially satisfied. Next, suppose that the lemma is true for all polynomials of degree strictly less than the degree of $f$. In particular, it is true for $\w f := (\del/\del x_i)[f]$, where $i = 1,\ldots,d$ is chosen so that $\w f$ is nonzero. Let $\w\gamma > 0$ be the value given by the induction hypothesis. Since $L := Z_f\butnot Z_{\w f}$ is a smooth $(k - 1)$-dimensional submanifold of $\R^k$, for all $x\in L$ and for all sufficiently small neighborhoods $B$ of $x$, $L\cap B$ is contained in a small neighborhood of a hyperplane. Thus since $K := Z_f\cap B(\0,2)\butnot \Int(Z_{\w f}^{(\w\gamma/2)})$ is a compact subset of $L$, there exists $\delta > 0$ with the following property:
\begin{equation}
\label{deltadef}
\begin{split}
&\text{For every ball $B(\xx,\rho)\subset\R^k$ satisfying $0 < \rho \leq \delta$,}\\
&\text{there exists a hyperplane $\LL\subset\R^k$ such that $K\cap B(\xx,2\rho)\subset \LL^{(\beta\rho/2)}$.}
\end{split}
\end{equation}
Let $\gamma = \beta^2\min(\w\gamma,\delta)/2$.

Alice's strategy is now as follows: Use the strategy from the induction hypothesis to guarantee that Bob's first ball of radius less than $\w\gamma$ is disjoint from $Z_{\w f}^{(\w \gamma)}$. If the radius of this ball is greater than $\delta$, make further moves arbitrarily until Bob chooses a ball of radius less than $\delta$. Either way, let $B = B(\xx,\rho)$ denote Bob's first ball satisfying $\rho\leq \min(\w\gamma,\delta)$, and note that $\rho\geq\beta\min(\w\gamma,\delta) = 2\gamma/\beta$. In particular $\rho > \gamma$, so Bob has not yet chosen a ball of radius less than $\gamma$. Let $\LL$ be a hyperplane such that $K\cap B(\xx,2\rho)\subset \LL^{(\beta\rho/2)}$, guaranteed to exist by \eqref{deltadef}. Alice's next move will be to delete the $\beta\rho$-neighborhood of the hyperplane $\LL$. Following that, she will make arbitrary moves until Bob chooses a ball $\w B$ of radius less than $\gamma$.

We claim that $\w B$ is disjoint from $Z_f^{(\gamma)}$. Indeed, fix $\yy\in \w B\subset B\butnot \LL^{(\beta\rho)}$. Then for $\zz\in Z_f$, either
\begin{itemize}
\item[(1)] $\zz\in K\cap B(\xx,2\rho)\subset \LL^{(\beta\rho/2)}$, in which case \[\|\zz - \yy\|\geq \dist(\LL^{(\beta\rho/2)},\R^k\butnot\LL^{(\beta\rho)}) = \beta\rho/2\geq\gamma,\]
\item[(2)] $\zz\notin B(\xx,2\rho)$, in which case \[\|\zz - \yy\| \geq \dist\big(\R^k\butnot B(\xx,2\rho),B(\xx,\rho)\big) = \rho \geq \gamma,\]
\item[(3)] $\zz\notin B(\0,2)$, in which case \[\|\zz - \yy\| \geq \dist\big(\R^k\butnot B(\0,2),B(\0,1)\big) = 1 \geq \gamma,\text{ or}\]
\item[(4)] $\zz\in Z_f\cap B(\0,2)\butnot K\subset Z_{\w f}^{(\w\gamma/2)}$, in which case \[\|\zz - \yy\| \geq \dist(Z_{\w f}^{(\w\gamma/2)},\R^k\butnot Z_{\w f}^{(\w\gamma)}) \geq \w\gamma/2 \geq \gamma.\]
\end{itemize}
Thus $\yy\notin Z_f^{(\gamma)}$.
\end{proof}

We next show that the constant $\gamma$ can be made to depend only on the degree of $f$ and not on $f$ itself.

\begin{lemma}
\label{lemmaavoidvariety2}
Fix $k,D\in\N$ and $0 < \beta\leq 1$. There exists $\gamma > 0$ such that for any nonzero polynomial $f:\R^k\to\R$ of degree at most $D$, if Bob and Alice play the $\beta$-hyperplane game and if Bob's first move is $B_0 = B(\0,1)$, then Alice has a strategy to guarantee that Bob's first ball of radius less than $\gamma$ (assuming that such a ball exists) is disjoint from $Z_f^{(\gamma)}$.
\end{lemma}
\begin{proof}
The map $\PP_D\ni f\mapsto Z_f$ is upper semicontinuous in the Vietoris topology (cf. \cite[\64.F]{Kechris}), meaning that for any $f\in \PP_D$, $\gamma > 0$, and $K\subset\R^k$ compact, there exists a neighborhood of $f$ in $\PP_D$ such that all $g$ in the neighborhood satisfy $Z_g\cap K\subset Z_f^{(\gamma)}$. In particular, for each $f\in \PP_D$, let $\gamma_f$ be as in Lemma \ref{lemmaavoidvariety}, and let $U_f\subset \PP_D$ be a neighborhood of $f$ such that for all $g\in U_f$, $Z_g\cap B(\0,2)\subset Z_f^{(\gamma_f/2)}$. Let $(U_{f_i})_{i = 1}^n$ be a finite subcover (which exists since $\PP_D$ is compact) and let $\gamma = \min_{i = 1}^n \gamma_{f_i}/2$. Then for all $g\in \PP_D$, $g\in U_{f_i}$ for some $i$, and so
\[
Z_g^{(\gamma)} \cap B(\0,1)\subset Z_{f_i}^{(\gamma_{f_i}/2 + \gamma)} \subset Z_{f_i}^{(\gamma_{f_i})}.
\]
Since Alice has a strategy to avoid $Z_{f_i}^{(\gamma_{f_i})}$ by the time Bob's radius is less than $\gamma_{f_i}$, she has a strategy to avoid $Z_g^{(\gamma)}$ by the time Bob's radius is less than $\gamma$.
\end{proof}

Let $k$, $D$, $\beta$, and $\gamma$ be as above. Fix $\xx\in\R^k$ and $\rho > 0$, and let
\begin{equation}
\label{Txrho}
T_{\xx,\rho}(\ww) = \xx + \rho\ww,
\end{equation}
so that $T_{\xx,\rho}\big(B(\0,1)\big) = B(\xx,\rho)$. Translating Lemma \ref{lemmaavoidvariety} via the map $T_{\xx,\rho}$, we see that if Bob and Alice play the $\beta$-hyperplane game, then after Bob makes a move $B(\xx,\rho)$, Alice may devote the next several turns to ensuring that Bob's first ball of radius less than $\gamma\rho$ is disjoint from $Z_{f\circ T_{\xx,\rho}^{-1}}^{(\gamma\rho)}$. This allows her to translate any winning strategy for the $(\beta\gamma,D)$ algebraic-set game into a winning strategy for the $\beta$-hyperplane game. Indeed, if in the algebraic-set game Alice responds to Bob's move $B(\xx,\rho)$ by deleting the set $Z_f^{(\gamma\rho)}$, then in the $\beta$-hyperplane game Alice simply spends the next several turns avoiding $Z_f^{(\gamma\rho)}$. Bob's first ball of radius less than $\gamma\rho$ will still have radius $\geq \beta\gamma\rho$ by the rules of the $\beta$-hyperplane game, so it can be interpreted as Bob's next move in the $(\beta\gamma,D)$ algebraic set game. Summarizing, we have the following:

\begin{corollary}
\label{corollaryAWimpliesHW}
Any algebraic-set winning set is hyperplane winning.
\end{corollary}
\begin{proof}
For each $k,D\in\N$ and $0 < \beta\leq 1$, if $\gamma > 0$ is as in Lemma \ref{lemmaavoidvariety2}, then every $(\beta\gamma,D)$ algebraic-set winning subset of $\R^k$ is $\beta$-hyperplane winning.
\end{proof}

To complete the proof of Theorem \ref{theoremwinningequivalence}, we must show that every levelset winning set is algebraic-set winning. For this, we will need three more lemmas:

\begin{lemma}
\label{lemmafgamma}
Fix $k,D\in\N$ and $\beta > 0$. Then there exists $\gamma > 0$ such that for any $f\in \PP_D$, there exists $g\in \PP_D$ such that
\[
f^{-1}(-\gamma,\gamma)\cap B(\0,1)\subset Z_g^{(\beta)}.
\]
\end{lemma}
\begin{proof}
For each $g\in \PP_D$, $|g|$ is bounded uniformly away from $0$ on $B(\0,1)\butnot Z_g^{(\beta)}$. Let $\gamma_g > 0$ be strictly less than this uniform bound, and let $U_g$ be the set of all polynomials $f\in \PP_D$ such that $\min_{B(\0,1)\butnot Z_g^{(\beta)}}|f| > \gamma_g$. Then $U_g$ is an open set containing $g$. Letting $(U_{g_i})_{i = 1}^n$ be a finite subcover, the lemma holds with $\gamma = \min_{i = 1}^n \gamma_{g_i}$.
\end{proof}

\begin{lemma}
\label{lemmaapproximatelevelset1}
Fix $k,D\in\N$ and $\beta > 0$, and let $B = B(\0,1)$. There exists $\delta > 0$ such that if $f:B\to\R$ is a $\CC^{D + 1}$ function satisfying
\[
\sup_B |f^{(\alpha)}| \leq \delta\|f\|_{\CC^D,B} \all \alpha\in\Neur^k \text{ with }|\alpha| = D + 1,
\]
then there exists $g\in \PP_D$ such that
\[
Z_f \subset f^{-1}(-\delta\|f\|_{\CC^D,B},\delta\|f\|_{\CC^D,B}) \subset Z_g^{(\beta)}.
\]
In particular, $Z_f^{(\beta)} \subset Z_g^{(2\beta)}$.
\end{lemma}
\begin{proof}
Fix $\delta > 0$ small to be determined, and let $f:B\to\R$ be as above. For convenience of notation, we without loss of generality assume that $\|f\|_{\CC^D,B} = 1$. By the definition of $\|f\|_{\CC^D,B}$, there exists a point $\zz\in B$ such that $\|f\|_{\CC^D,\zz} \geq 1/2$. Let $h_\zz$ denote the $D$th order Taylor polynomial for $f$ centered at $\zz$. Then
\begin{equation}
\label{hzB}
\|h_\zz\|_{\CC^D,B} \geq \|h_\zz\|_{\CC^D,\zz} = \|f\|_{\CC^D,\zz} \geq1/2\,.
\end{equation}
By Taylor's theorem, for all $\xx\in B$
\[
|f(\xx) - h_\zz(\xx)| \lesssim_\times \max_{|\alpha| = D + 1} \sup_B |f^{(\alpha)}| \cdot \|\xx - \zz\|^{D + 1} \lesssim_\times \delta\,,
\]
and so
\[
|h_\zz(\xx)|\lesssim_\times \delta \all \xx\in f^{-1}(-\delta,\delta)\,.
\]
Write $h_\zz = cj$ for some $c > 0$ and $j\in \PP_D$; then $\|j\|_{\CC^D,B} \asymp_\times 1$ since $\PP_D$ is compact. Combining with \eqref{hzB}, we see that $c \gtrsim_\times 1$, and thus
\begin{equation}
\label{jxbound}
|j(\xx)|\lesssim_\times \delta \all \xx\in f^{-1}(-\delta,\delta).
\end{equation}
Let $\gamma > 0$ be as in Lemma \ref{lemmafgamma}, and let $\delta$ be $\gamma$ divided by the implied constant of \eqref{jxbound}. Then
\[
f^{-1}(-\delta,\delta) \subset j^{-1}(-\gamma,\gamma).
\]
Moreover, by Lemma \ref{lemmafgamma} there exists $g\in \PP_D$ such that $j^{-1}(-\gamma,\gamma) \subset Z_g^{(\beta)}$. This completes the proof.
\end{proof}

\begin{lemma}
\label{lemmaapproximatelevelset2}
Fix $k,D\in\N$ and $\beta,C_1 > 0$. Then there exists $\epsilon > 0$ such that for any ball $B = B(\xx,\rho)\subset\R^k$ satisfying $\rho\leq\epsilon$ and for any $\CC^{D + 1}$ function $f:B\to\R$ satisfying
\[
\|f\|_{\CC^{D + 1},B} \leq C_1\|f\|_{\CC^D,B},
\]
there exists a polynomial $g:\R^k\to\R$ of degree at most $D$ such that $Z_f\subset Z_g^{(\beta\rho)}$, and thus $Z_f^{(\beta\rho)} \subset Z_g^{(2\beta\rho)}$.
\end{lemma}
\begin{proof}
Fix $0 < \epsilon \leq 1$ small to be determined, and let $B = B(\xx,\rho)$ and $f:B\to\R$ be as above. Let $T_{\xx,\rho}$ be given by \eqref{Txrho}, and let $\w f = f\circ T_{\xx,\rho}$. Then for all $\alpha\in\Neur^k$ with $|\alpha| = D + 1$,
\[
\sup_{B(\0,1)} |\w f^{(\alpha)}| = \rho^{D + 1} \sup_{B(\0,1)} |f^{(\alpha)}\circ T_{\xx,\rho}| \leq \rho^{D + 1}\|f\|_{\CC^{D + 1},B} \lesssim_\times \rho^{D + 1}\|f\|_{\CC^D,B},
\]
and on the other hand
\[
\|\w f\|_{\CC^D,B(\0,1)} = \max_{|\alpha| \leq D}\sup_{B(\0,1)} \w f^{(\alpha)} = \max_{|\alpha| \leq D} \rho^{|\alpha|} \sup_{B(\0,1)} |f^{(\alpha)}\circ T_{\xx,\rho}| \geq \rho^D \|f\|_{\CC^D,B}.
\]
Combining, we have
\[
\sup_{B(\0,1)} |\w f^{(\alpha)}| \lesssim_\times \rho\|\w f\|_{\CC^D,B(\0,1)} \all \alpha\in\Neur^k \text{ with }|\alpha| = D + 1.
\]
So for $\epsilon$ sufficiently small, $\w f$ satisfies the hypotheses of Lemma \ref{lemmaapproximatelevelset1}. Let $\w g$ be the polynomial given by Lemma \ref{lemmaapproximatelevelset1}, and let $g = \w g\circ T_{\xx,\rho}^{-1}$, so that $Z_g = T_{\xx,\rho}(Z_{\w g})$. This completes the proof.
\end{proof}

Let $k$, $D$, $\beta$, $C_1$, and $\epsilon$ be as above. Lemma \ref{lemmaapproximatelevelset2} gives us a way of translating a winning strategy for Alice in the $(\beta,D,C_1)$-levelset game into a winning strategy for Alice in the $(2\beta,D)$ algebraic-set game. Indeed, without loss of generality suppose that Bob's first move in the $(\beta,D,C_1)$-levelset game has radius $\leq\epsilon$. (Otherwise Alice makes dummy moves until this is true.) Now if Alice responds to Bob's move $B(\xx,\rho)$ in the $(\beta,D,C_1)$-levelset game by deleting the set $Z_f^{(\beta\rho)}$, then in the $(2\beta,D)$ algebraic-set game, she will simply delete the set $Z_g^{(2\beta\rho)}$, where $g$ is given by Lemma \ref{lemmaapproximatelevelset2}. Summarizing, we have the following:

\begin{corollary}
\label{corollaryLWimpliesAW}
Any levelset winning set is algebraic-set winning.
\end{corollary}
\begin{proof}
For each $k,D\in\N$ and $\beta,C_1 > 0$, then every $(\beta,D,C_1)$-levelset winning subset of $\R^k$ is $(2\beta,D)$ algebraic-set winning.
\end{proof}

Combining Corollaries \ref{corollaryAWimpliesHW} and \ref{corollaryLWimpliesAW} completes the proof of Theorem \ref{theoremwinningequivalence}.

\section{The simplex lemma and its consequences}
\label{sectionsimplex}

The paradigmatic example of a hyperplane winning set is the set
$
\BA_d
$  defined by the formula \equ{defbad},
% = \BA({\psi_{1 + 1/d},\R^d, \Q^d,H})\,,
%\]
which was proven to be hyperplane winning in \cite[Theorem 2.5]{BFKRW}, as a consequence of the so-called \emph{simplex lemma} \cite[Lemma 3.1]{BFKRW}. Essentially, the simplex lemma states that for each ball $B(\xx,\rho)\subset\R^d$, the set of rational points in $B(\xx,\rho)$ whose denominators are less than $\epsilon \rho^{-d/(d + 1)}$ is contained in an affine hyperplane, where $\epsilon > 0$ is small and depends only on $d$. As a result, when playing the hyperplane game Alice can simply delete the neighborhood of the hyperplane given by the simplex lemma, and it turns out that this strategy is winning for $\BA_d$. In this section we prove an analogue of the simplex lemma for rational points in a fixed manifold $M$. We then use the simplex lemma to prove two general negative results about intrinsic approximation on manifolds: that $\BA_M(\psi_{c(k,d)})$ is hyperplane winning, and that $\lambda_M(\VWA_M) = 0$.

% In the case $M = \R^d$, the fact that $\BA_{\psi_{d,d}}$ is hyperplane winning is more or less a direct consequence of the Simplex Lemma; see Theorem 2.5 and/or Lemma 3.1 of \cite{BFKRW}. Similarly, we will use an analogue of the Simplex Lemma to prove Theorem \ref{theoremBAwinning}.

%Recall that for $1\leq k\leq d$, the constants $N_{k,d}$ and $c(k,d) = (d + 1)/N_{k,d}$ were defined in Notation \ref{notationkd}.

\begin{lemma}[Simplex lemma for intrinsic approximation on manifolds]
\label{lemmasimplex}
Let $M\subset\R^d$ be a submanifold of dimension $k$, let $\Psi:U\to M$ be a local parameterization of $M$, and let $V\subset U$ be compact. Then there exists $\kappa > 0$ such that for all $\ss\in U$ and $0 < \rho\leq 1$, the set
\[
S_{\ss,\rho} := \{\pp/q\in\Q^d\cap \Psi\big(V\cap B(\ss,\rho)\big): q\leq \kappa \rho^{-1/c(k,d)}\}
\]
is contained in a hyperplane.
%If $\kappa$ is sufficiently small, then all rational points in $\Psi(2B_n)$ whose denominators lie in the $n$th window lie in a hyperplane.
\end{lemma}
%\begin{remark}
%In the original Simplex Lemma, the set $B(\xx,\kappa \numberq^{-(d + 1)/N_{k,d}})$ could be replaced with any convex set of volume less than $\kappa \numberq^{-(d + 1)}$. 
%\end{remark}
\begin{proof}%[Proof of Lemma \ref{lemmasimplex}]
For all $\ss\in U$ let $\emb(\ss) = (1,\Psi(\ss))$. Define a function $f:U^{d + 1}\to\R$ by
\[
f(\ss_1,\ldots,\ss_{d + 1}) =
\det[\emb(\ss_1) \; \cdots \; \emb(\ss_{d + 1})].
\]
Then $f$ vanishes along the diagonal
\[
\Delta = \{(\tt,\ldots,\tt):\tt\in U\}.
\]
In fact, the first several derivatives of $f$ vanish along the diagonal, due to repeated columns:
\begin{claim}
\label{claimNkd}
The smallest order derivative of $f$ which does not vanish identically along the diagonal is no less than $N_{k,d}$.
\end{claim}
\begin{proof}
Suppose that for some sequence of multi-indices $\alpha_1,\ldots,\alpha_{d + 1}\in\Neur^k$, the expression
\begin{equation}
\label{alphas}
\frac{\del}{\del^{\alpha_1}\tt_1}\cdots\frac{\del}{\del^{\alpha_{d + 1}}\tt_{d + 1}}f(\tt_1,\ldots,\tt_{d + 1})
\end{equation}
does not vanish identically along the diagonal $\{\tt_1 = \cdots = \tt_{d + 1}\}$. Here we use the multi-index notation
\[
\frac{\del}{\del^{\alpha_i} \tt_i} = \left(\frac{\del}{\del t_{i,1}}\right)^{\alpha_{i,1}} \cdots \left(\frac{\del}{\del t_{i,k}}\right)^{\alpha_{i,k}}.
\]
Since the determinant of a matrix is linear with respect to the columns of that matrix, we have
\[
\frac{\del}{\del^{\alpha_1}\tt_1}\cdots\frac{\del}{\del^{\alpha_{d + 1}}\tt_{d + 1}}f(\tt_1,\ldots,\tt_{d + 1})
= \det[\del^{\alpha_1}\emb(\tt_1) \; \cdots \; \del^{\alpha_{d + 1}}\emb(\tt_{d + 1})].
\]
Since this does not vanish identically along the diagonal, there exists $\tt\in U$ such that
\[
\det[\del^{\alpha_1}\emb(\tt) \; \cdots \; \del^{\alpha_{d + 1}}\emb(\tt)] \neq 0.
\]
In particular, the rows $\left(\del^{\alpha_i}\emb(\tt)\right)_{i = 1}^{d + 1}$ are all distinct, so the multi-indices $\alpha_1,\ldots,\alpha_{d + 1}$ must be distinct. Thus for each $j\in\N$,
\begin{equation}
\label{njbounds}
n_j := \#\{i = 1,\ldots,d + 1:|\alpha_i| = j\} \leq \#\{\alpha\in\Neur^k:|\alpha| = j\} = [k - 1,j],
\end{equation}
and on the other hand,
\begin{equation}
\label{njsum}
\sum_{j = 0}^\infty n_j = d + 1.
\end{equation}
The order of the derivative \eqref{alphas} is
\[
\sum_{i = 1}^{d + 1}|\alpha_i| = \sum_{j = 1}^\infty j n_j,
\]
so computing the smallest order derivative of $f$ which potentially does not vanish along the diagonal becomes a combinatorial problem of minimizing $\sum_{j = 1}^\infty j n_j$ subject to \eqref{njbounds} and \eqref{njsum}. The reader will verify that the minimum is attained when
\[
n_j = \begin{cases}
[k - 1,j] & \text{if } j < n_{k,d} + 1\\
m_{k,d} & \text{if } j = n_{k,d} + 1\\
0 & \text{if } j > n_{k,d} + 1
\end{cases} \;\; (j \geq 0),
\]
and that the value of $\sum_{j = 1}^\infty j n_j$ on this sequence is $N_{k,d}$, where $n_{k,d}$, $m_{k,d}$, and $N_{k,d}$ are as in Notation \ref{notationkd}.
\QEDmod\end{proof}
Thus by Taylor's theorem and the compactness of $V$, we have
\begin{equation}
\label{taylor}
|f(\ss_1,\ldots,\ss_{d + 1})| \lesssim_\times \dist\left((\ss_i)_1^{d + 1},\Delta\right)^{N_{k,d}}
\end{equation}
for all $\ss_1,\ldots,\ss_{d + 1}\in V$.

Now by contradiction, suppose that the points $\rr_1,\ldots,\rr_{d + 1}\in S_{\ss,\rho}$ do not lie in a hyperplane. For each $i$ write $\rr_i = \Psi(\ss_i) = \pp_i/q_i$. Let
\[
D = f(\ss_1,\ldots,\ss_{d + 1}) = \det\left[\begin{array}{ccc}
1 & \cdots & 1\\
\rr_1 & \cdots & \rr_{d + 1}
\end{array}\right] \neq 0.
\]
Since $\ss_i\in B(\ss,\rho)$, we have $\dist\left((\ss_i)_1^{d + 1},\Delta\right) \lesssim_\times \rho$. Thus by \eqref{taylor} we have
\begin{equation}
\label{taylor2}
|D|\lesssim_\times \rho^{N_{k,d}}.
\end{equation}
On the other hand, we have
\[
D = \prod_{i = 1}^{d + 1} \frac{1}{q_i} \det\left[\begin{array}{ccc}
q_1 & \cdots & q_{d + 1}\\
\pp_1 & \cdots & \pp_{d + 1}
\end{array}\right] \in \prod_{i = 1}^{d + 1} \frac{1}{q_i} \Z.
\]
Thus,
\[
|D| \geq \frac{1}{\prod_{i = 1}^{d + 1} q_i}\cdot
\]
Since by assumption $q_i\leq \kappa \rho^{-N_{k,d}/(d + 1)}$, we have
\[
|D| \geq \kappa^{-(d + 1)}\rho^{N_{k,d}}.
\]
For $\kappa > 0$ sufficiently small, this contradicts \eqref{taylor2}.
%Combining with \eqref{simplexlemma} gives
%\[
%\frac{1}{(\kappa\delta^{n + 1})^{d + 1}} \lesssim_\times \rho_n^{N_{k,d}} \asymp_\times (\gamma^n)^{N_{k,d}} = \delta^{-n(d + 1)},
%\]
%which is a contradiction for $\kappa > 0$ sufficiently small.
\end{proof}

Using the simplex lemma, we proceed to prove two results about intrinsic approximation on $M$. The first is the following:

\begin{theorem}
\label{theoremBAwinning}
Let $M\subset\R^d$ be a submanifold of dimension $k$, and let $c(k,d)$ be as in Notation \ref{notationkd}. Suppose that for some $D\in\N$, every point of $M$ is $D$-nondegenerate. Then $\BA_M(\psi_{c(k,d)})$ is hyperplane winning relative to $M$.
\end{theorem}
\begin{proof}
Let $\Psi:U\to M$ be a local parameterization of $M$, and let $K\subset U$ be compact. We need to show that the set
\begin{equation}
\label{BAMK}
\Psi^{-1}\big(\BA_M(\psi_{c(k,d)})\big)\cup (\R^k\butnot K)
\end{equation}
is hyperplane winning.

Fix $C_1 > 0$ large to be determined, and let $\beta > 0$. We will show that the set \eqref{BAMK} is $(\beta,D,C_1)$-levelset winning, where $D$ is as in the statement of Theorem \ref{theoremBAwinning}. Let $\lambda = \beta^{-1/c(k,d)}$ (so that $\lambda > 1$). Denote Bob's first move by $B_0 = B(\ss_0,\rho_0)\subset \R^k$. Fix an open set $V\supset K$ which is relatively compact in $U$; without loss of generality we may assume that $B(\ss_0,2\rho_0)\subset V$, since Alice may make dummy moves until either this is true or Bob's ball is disjoint from $K$. Now Alice's strategy is as follows: If Bob has just made his $n$th move $B_n = B(\ss_n,\rho_n) \subset V$, then Alice will delete the $\beta\rho_n$-neighborhood of the set $\Psi^{-1}(\LL_n)$, where $\LL_n$ is the affine hyperplane containing the set $S_{\ss_n,2\rho_n}$. To complete the proof we need to show (i) that this is legal (given $C_1 > 0$ large enough), and (ii) that the strategy guarantees that $\bigcap_1^\infty B_n \cap \Psi^{-1}\big(\BA_M(\psi_{c(k,d)})\big) \neq \emptyset$.

\begin{itemize}
\item[(i)] For each $\ss\in U$ let $\emb(\ss) = (1,\Psi(\ss))$. Since every point of $M$ is $D$-nondegenerate, for each $\ss\in U$ and $\ww\in \R^{d + 1}\butnot\{\0\}$ we have
\[
\|\tt\mapsto \ww\cdot\emb(\tt)\|_{\CC^D,\ss} > 0,
\]
and by continuity, this quantity is bounded from below uniformly for $\ss\in V$ and $\ww\in S^d$.

Now consider Alice's $n$th move. Write $\LL_n = \{\xx\in\R^d: \ww\cdot (1,\xx) = 0\}$ for some $\ww\in S^d$. Define $f:U\to\R$ by
\[
f(\ss) = \ww\cdot\emb(\ss),
\]
so that $Z_f = \Psi^{-1}(\LL_n)$. Then by the first paragraph, $\|f\|_{\CC^D,B_n} \geq \|f\|_{\CC^D,\ss_n}$ is bounded from below. On the other hand,
\[
\|f\|_{\CC^{D + 1},B_n} \leq \|\ww\|\cdot\|\emb\|_{\CC^{D + 1},B_n} \leq \|\emb\|_{\CC^{D + 1},V} \asymp_\times 1,
\]
so $\|f\|_{\CC^{D + 1},B_n}\lesssim_\times \|f\|_{\CC^D,B_n}$. Letting $C_1$ be the implied constant finishes the proof.

\item[(ii)] By Remark \ref{remarksingleton}, we may assume that $\bigcap_1^\infty B_n$ is a singleton, say $\bigcap_1^\infty B_n = \{\ss\}$. For each $\rr = \pp/q \in\Q^d\cap M$, let $n\in\N$ be minimal such that $q\leq \kappa\rho_n^{-N_{k,d}/(d + 1)}$. If $\rr\in \Psi\big(V\cap B(\ss_n,2\rho_n)\big)$, then by Lemma \ref{lemmasimplex} we have $\rr\in\LL_n$. Since Alice deleted the set $\Psi^{-1}(\LL_n)^{(\beta\rho_n)}$, we have $\ss\notin \Psi^{-1}(\LL_n)^{(\beta\rho_n)}$ and thus
\[
\dist(\Psi(\ss),\rr) \asymp_\times \|\ss - \Psi^{-1}(\rr)\| \geq \beta\rho_n \asymp_{\times,\beta,\kappa} q^{-(d + 1)/N_{k,d}}.
\]
On the other hand, if $\rr\notin \Psi\big(V\cap B(\ss_n,2\rho_n)\big)$, then either $\rr\notin\Psi(V)$, which implies
\[
\dist(\Psi(\ss),\rr) \geq \dist\big(\Psi(\ss),M\butnot \Psi(V)\big) \asymp_\times 1,
\]
or $\rr\in\Psi\big(V\butnot B(\ss_n,2\rho_n)\big)$, in which case
\[
\dist(\Psi(\ss),\rr) \asymp_\times \|\ss - \Psi^{-1}(\rr)\| \geq \rho_n \asymp_{\times,\beta,\kappa} q^{-(d + 1)/N_{k,d}}.
\]
In all cases we have $\dist(\Psi(\ss),\rr) \gtrsim_\times \psi_{c(k,d)}(q)$, so $\Psi(\ss)\in\BA_M(\psi_{c(k,d)})$.
\end{itemize}
\end{proof}

To state our last theorem regarding general manifolds, we need a definition:

%\begin{definition}
%A measure $\mu$ on $M$ will be called \emph{intrinsically extremal} if the set $\VWA_M$ defined in \eqref{VWAdef} has $\mu$-measure zero.
%\end{definition}

\begin{definition}
\label{definitionAF}
A measure $\mu$ on an open set $U\subset\R^k$ is \emph{absolutely decaying} if there exist $C,\alpha > 0$ such that for all $\xx\in\Supp(\mu)$, for all $0 < \rho\leq 1$ such that $B(\xx,\rho)\subset U$, for all $\epsilon > 0$, and for every affine hyperplane $\LL\subset\R^k$, we have
\[
\mu\left(\LL^{(\epsilon\rho)}\cap B(\xx,\rho)\right) \leq C \epsilon^\alpha\mu\big(B(\xx,\rho)\big).
\]
$\mu$ is called \emph{doubling} if $\mu\big(B(\xx,2\rho)\big) \asymp_\times \mu\big(B(\xx,\rho)\big)$ for all $\xx\in\Supp(\mu)$ and $0 < \rho \leq 1$. If $\mu$ is both absolutely decaying and doubling, then $\mu$ is called \emph{absolutely friendly}.
\end{definition}

%Recall that the set $\VWA_M$ is defined by the equation
%\[
%\VWA_M := \bigcup_{c > c(k,d)}\A_M(\psi_c).
%\]

\begin{theorem}
\label{theoremextremal}
Let $M\subset\R^d$ be a submanifold of dimension $k$, and fix $D\in\N$. Let $\Psi:U\to M$ be a local parameterization of $M$, let $\mu$ be an absolutely friendly measure on $U$, and let $\nu = \Psi[\mu]$. If $\nu$-a.e.\ point of $M$ is $D$-nondegenerate, then $\VWA_M(\psi_{c(k,d)})$ is a $\nu$-nullset. In particular, $\lambda_M\big(\VWA_M(\psi_{c(k,d)})\big) = 0$.
\end{theorem}

%\comdmitry{2. So do you think you can upgrade the statement to a Hausdorff measure estimate, at least for $\lambda_M$? The argument is quite similar to the paper by Pollington and Velani which does it for $\R^d$, by the way it will be nice to refer to it.}

\begin{proof}
Let $\kappa > 0$ be as in Lemma \ref{lemmasimplex}. Fix $\lambda > 1$ arbitrary, and for each $n\in\N$ let
\begin{align*}
\numberq_n &:= \kappa \lambda^{n/c(k,d)}\\
\rho_n &:= \frac12\lambda^{-n}.
\end{align*}
Let $K\subset U$ be a compact set, and let $V\supset K$ be open and relatively compact in $U$. Then by Lemma \ref{lemmasimplex}, we have the following:
\begin{corollary} \label{corollarysimplex}
For all $\ss\in V$ and for all $n\in\N$, the set
\[
S_{n,\ss} = \left\{\pp/q\in \Q^d\cap \Psi\big(V\cap B(\ss,2\rho_n)\big): q\leq \numberq_n\right\}
\]
is contained in a hyperplane.
\end{corollary}
Denote the hyperplane guaranteed by Corollary \ref{corollarysimplex} by $\LL_{n,\ss}$. For each $n\in\N$, let $(\ss_i^{(n)})_{i = 1}^{N_n}$ be a maximal $\rho_n$-separated subset of $K$. Then $\{B(\ss_i^{(n)},\rho_n) : i = 1,\ldots,N_n\}$ is a cover of $K$ whose multiplicity is bounded depending only on $d$. For each $i = 1,\ldots,N_n$, let $\LL_{n,i} = \LL_{n,\ss_i^{(n)}}$.
\begin{claim}
\label{claimSbound}
\[
\Psi^{-1}\big(\VWA_M(\psi_{c(k,d)})\big)\cap K \subset \lim_{\gamma\to 0} \limsup_{n\to\infty}\bigcup_{i = 1}^{N_n}\left[\left(\Psi^{-1}(\LL_{n,i})\right)^{(\rho_n^{1 + \gamma})}\cap B(\ss_i^{(n)},\rho_n)\right].
\]
\end{claim}
\begin{subproof}
Fix $\ss \in \Psi^{-1}\big(\VWA_M(\psi_{c(k,d)})\big)\cap K$, and recall the definition of the set $\VWA_M(\psi_{c(k,d)})$ given by \equ{defvwa} and \equ{defvwam}. Since $\Psi(\ss)\in\VWA_M(\psi_{c(k,d)})$, there exists $\epsilon > 0$ such that there are infinitely many $\rr = \pp/q\in\Q^d\cap M$ satisfying
\begin{equation}
\label{pq}
\|\Psi(\ss) - \rr\| \leq q^{-(c(k,d) + \epsilon)}.
\end{equation}
Fix such an $\rr$, and let $n\in\N$ satisfy $\numberq_{n - 1}\leq q < \numberq_n$. Then
\[
\|\ss - \Psi^{-1}(\rr)\| \asymp_\times \|\Psi(\ss) - \rr\| \leq \numberq_{n - 1}^{-(c(k,d) + \epsilon)}
\asymp_\times \rho_n^{1 + \epsilon/c(k,d)}.
\]
Fix $0 < \gamma < \epsilon/c(k,d)$. If $n$ is sufficiently large, then we have
\[
\|\ss - \Psi^{-1}(\rr)\| \leq \rho_n^{1 + \gamma} \leq \rho_n.
\]
On the other hand, since $\ss\in K$, we have $\ss\in B(\ss_i^{(n)},\rho_n)$ for some $i = 1,\ldots,N_n$. It follows that $\rr\in \Psi\big(V\cap B(\ss_i^{(n)},2\rho_n)\big)$, and so by Corollary \ref{corollarysimplex} we have $\rr\in\LL_{n,i}$. Thus
\[
\ss \in \left(\Psi^{-1}(\LL_{n,i})\right)^{(\rho_n^{1 + \gamma})}\cap B(\ss_i^{(n)},\rho_n).
\]
Since this argument holds for all $\rr$ satisfying (\ref{pq}), it follows that
\[
\ss \in \bigcup_{i = 1}^{N_n}\left[\left(\Psi^{-1}(\LL_{n,i})\right)^{(\rho_n^{1 + \gamma})}\cap B(\ss_i^{(n)},\rho_n)\right]
\]
for infinitely many $n\in\N$.
\end{subproof}
\begin{claim}
\label{claimgammaalpha}
For each $\gamma > 0$, there exists $\alpha > 0$ such that for all $n\in\N$ and $i = 1,\ldots,N_n$,
\begin{equation}
\label{mubound}
\mu\left(\Psi^{-1}(\LL_{n,i})^{(\rho_n^{1 + \gamma})}\cap B(\ss_i^{(n)},\rho_n)\right)
\lesssim_\times \rho_n^\alpha \mu\big(B(\ss_i^{(n)},\rho_n)\big).
\end{equation}
\end{claim}
\begin{subproof}
For each $\ss\in U$ let $\emb(\ss) = (1,\Psi(\ss))$. Let $B = B(\ss_i^{(n)},\rho_n)$. By \cite[Proposition 7.3]{KLW}, there exist $C,\alpha > 0$ such that for any linear map $P:\R^{d + 1}\to\R$, if $f = P\circ \emb$, then
\begin{equation}
\label{KLW}
\mu\left(f^{-1}\left(-\rho_n^{\gamma/2}\sup_B|f|,\rho_n^{\gamma/2}\sup_B|f|\right)\cap B\right) \leq \rho_n^\alpha \mu(B).
\end{equation}
On the other hand, if $P$ is the linear functional whose zero set is the hyperplane $\LL_{n,i}$, then
\[
\Psi^{-1}(\LL_{n,i}) = Z_f.
\]
So to complete the proof, we must show that
\begin{equation}
\label{ETSKLW}
Z_f^{(\rho_n^{1 + \gamma})}\cap B \subset f^{-1}\left(-\rho_n^{\gamma/2}\sup_B|f|,\rho_n^{\gamma/2}\sup_B|f|\right).
\end{equation}
Let $T = T_{\ss_i^{(n)},\rho_n}$ be as in \eqref{Txrho}, and let $\w f = f\circ T$. Translating \eqref{ETSKLW} via $T$ gives
\begin{equation}
\label{ETSKLW2}
Z_{\w f}^{(\rho_n^\gamma)}\cap B(\0,1) \subset \w f^{-1}\left(-\rho_n^{\gamma/2}\sup_{B(\0,1)}|\w f|,\rho_n^{\gamma/2}\sup_{B(\0,1)}|\w f|\right).
\end{equation}
To demonstrate \eqref{ETSKLW2}, we observe that if $\ss\in Z_{\w f}^{(\rho_n^\gamma)}\cap B(\0,1)$, then there exists $\tt\in Z_{\w f}$ for which $\|\ss - \tt\| \leq \rho_n^\gamma$. By Taylor's theorem, we have
\[
|\w f(\ss)| \lesssim_\times \|\w f\|_{\CC^1,B(\0,2)} \|\ss - \tt\| \leq \|\w f\|_{\CC^D,B(\0,2)} \rho_n^\gamma.
\]
So to complete the proof, we must show that
\begin{equation}
\label{ETSKLW3}
\|\w f\|_{\CC^D,B(\0,2)} \lesssim_\times \sup_{B(\0,1)}|\w f|.
\end{equation}
To demonstrate \eqref{ETSKLW3}, let $\beta > 0$ be small enough so that for every polynomial $g$ of degree at most $D$, $B(\0,1)\nsubset Z_g^{(2\beta)}$. Such a $\beta$ exists e.g.\ by a compactness argument. Let $\delta > 0$ be given by Lemma \ref{lemmaapproximatelevelset1}. For $n$ sufficiently large, the argument of Lemma \ref{lemmaapproximatelevelset2} shows that the hypotheses of Lemma \ref{lemmaapproximatelevelset1} are satisfied for $\w f$, and thus that
\[
\w f^{-1}(-\delta \|\w f\|_{\CC^D,B(\0,2)},\delta \|\w f\|_{\CC^D,B(\0,2)}) \subset Z_g^{(2\beta)} \propersubset B(\0,1).
\]
Thus there exists $\ss\in B(\0,1)$ for which $|\w f(\ss)| \geq \delta \|\w f\|_{\CC^D,B(\0,2)}$, demonstrating \eqref{ETSKLW3}.
\end{subproof}

Fix $\gamma,\alpha$ as in Claim \ref{claimgammaalpha}. From \eqref{mubound}, we see that
\begin{align*}
\sum_{i = 1}^{N_n}\mu\left(\big(\Psi^{-1}(\LL_{n,i})\big)^{(\rho_n^{1 + \gamma})}\cap B(\ss_i^{(n)},\rho_n)\right)
&\lesssim_\times \rho_n^{\alpha}\sum_{i = 1}^{N_n}\mu\big(B(\ss_i^{(n)},\rho_n)\big)
\lesssim_\times \rho_n^{\alpha}\mu(V)\\
\sum_{n = 0}^\infty\sum_{i = 1}^{N_n}\mu\left(\big(\Psi^{-1}(\LL_{n,i})\big)^{(\rho_n^{1 + \gamma})}\cap B(\ss_i^{(n)},\rho_n)\right)
&\lesssim_\times \sum_{n = 0}^\infty \lambda^{-\alpha n} < \infty.
\end{align*}
Thus by the Borel--Cantelli lemma, for each $\gamma > 0$
\[
\mu\left(\limsup_{n\to\infty}\bigcup_{i = 1}^{N_n}\left[\big(\Psi^{-1}(\LL_{n,i})\big)^{(\rho_n^{1 + \gamma})}\cap B(\ss_i^{(n)},\rho_n)\right]\right) = 0,
\]
and so $\mu\big(\VWA_M(\psi_{c(k,d)})\big)\ = 0$ by Claim \ref{claimSbound}.
\end{proof}

%\begin{remark}
%The conclusion of Theorem \ref{theoremextremal} holds for any measure $\mu$ satisfying \eqref{KLW}. In particular, in \cite{DFSU_geometrically_extremal_1}, a class of measures is considered which is vastly larger than the class of absolutely friendly measures, and these measures will be proven to satisfy \eqref{KLW}. Thus these measures will also satisfy $\mu(\VWA_M) = 0$. \comdavid{I think maybe we should delete this remark since it is not very precise and it would be annoying to add more details...}\internal
%\end{remark}

%\draftnewpage
\bibliographystyle{amsplain}

\bibliography{bibliography}

\end{document}